\documentclass[a4paper,reqno,11pt]{amsart}

\usepackage[utf8]{inputenc,xcolor}  %upside down !
\usepackage{marginnote}
\usepackage{geometry}
\usepackage{float} % H means exactly here
\usepackage[in]{fullpage}

\newcommand{\ds}{\displaystyle}

\newcommand{\tensor}{\otimes}

\newcommand{\op}{\mathcal}

\newcommand{\cdc}{,\dots,}

\newcommand{\FT}{\mathsf{ft}}

\newcommand{\wscmo}{\mathcal{H}_{\text{Lie}}}

\usepackage{amsmath}%
\usepackage{amsthm}
\usepackage{amsfonts}%
\usepackage{amssymb}%
\usepackage{graphicx}
\usepackage{xy,amsthm,enumerate,xypic,array}  %xypic,eufrak}

\input{xy}
\xyoption{all}

\numberwithin{equation}{section}

\newtheorem{theorem}{Theorem}[section]
\theoremstyle{plain}

\newtheorem{corollary}[theorem]{Corollary}
\newtheorem{lemma}[theorem]{Lemma}
\newtheorem{proposition}[theorem]{Proposition}

\theoremstyle{definition}
\newtheorem{definition}[theorem]{Definition}

\newtheorem{remark}[theorem]{Remark}

\allowdisplaybreaks[2]

\addtocounter{MaxMatrixCols}{2}

\begin{document}

\title{Wheel graph homology classes via Lie graph homology} 
\author{Benjamin C. Ward}
\email{benward@bgsu.edu}

%\date{December 20, 2010 }
%\subjclass{Primary , ?; Secondary ?, ?} 
%\keywords{Keyword one, keyword two etc.}

\begin{abstract} We give a new proof of the non-triviality of wheel graph homology classes using higher operations on Lie graph homology and a derived version of Koszul duality for modular operads.  \end{abstract}
\maketitle
%\tableofcontents

\section{Introduction.}

In the seminal paper 
\cite{WTw}, Willwacher constructed a family of non-trivial graph cohomology classes $\sigma_{2j+1}$ by analyzing an isomorphism between the $0^{th}$ cohomology of a certain graph complex $\mathsf{GC}_2$ with the Grothendieck-Teichm{\"u}ller Lie algebra $\mathfrak{grt}_1$.  The complex $\mathsf{GC}_2$ is (up to important details) the {\bf commutative} variant of a construction called the Feynman transform \cite{GeK2}.

In \cite{CHKV}, the authors study a family of group extensions of the outer automorphism groups of free groups and use the Leray-Serre spectral sequence to compute the homology of a low genus portion of the {\bf Lie} variant of this construction.  They show in particular that in genus 1, the virtual cohomological dimension consists of a family of classes $\alpha_{2j+1}$. %\footnote{they call it $\alpha_j$} 

The purpose of this paper is to show that the families of classes $\sigma_{2j+1}$ and $\alpha_{2j+1}$ correspond to each other under Koszul duality.

\subsection{The correspondence: non-technical version.}

Let us explain the nature of this correspondence and the perspective from which such a result is expected.  We omit several important technical details in this informal summary.

The statement that commutative and Lie structures are Koszul dual can be encoded in the language of operads.  One first defines an operadic generalization of the bar construction of an algebra.  One then proves that the homology of the bar construction of the commutative operad has the homotopy type of the Lie operad, and vice-versa.

Operads have operations parameterized by graphs of genus 0 (trees), and may be generalized by considering structures with operations parameterized by all graphs.  Such structures are called modular operads \cite{GeK2}, and the analog of the bar construction is called the Feynman transform which we denote $\FT$.  The commutative and Lie operads may be considered modular operads by simply declaring operations corresponding to higher genus graphs to be zero. % With the notion of Feynman transform we may describe the two families of graph homology classes above as living in $\sigma_{2j+1}\in H_\ast(\FTGK(\mathsf{Com}))$ and $\alpha_{2j+1}\in H_\ast(\FTGK(\mathsf{Lie}))$ respectively.

When viewed as modular operads the analogous Koszul duality relationship between $\mathsf{Com}$ and $\mathsf{Lie}$ no longer holds.  If it did, it would imply that $\FT(\mathsf{Com})$ had no homology in its higher genus summands which is patently false.   Indeed, both commutative graph homology $H_\ast(\FT(\mathsf{Com}))$  and Lie graph homology $H_\ast(\FT(\mathsf{Lie}))$ have a rich and highly non-trivial structure in higher genera.

There is, however, a more subtle Koszul duality relationship between these modular operads.  The modular operad $H_\ast(\FT(\mathsf{Lie}))$ encoding Lie graph homology contains a copy of the commutative modular operad $\mathsf{Com}$ but is Koszul dual (in a suitably derived sense) to the Lie modular operad $\mathsf{Lie}$.    Therefore in genus $g\geq 1$, the (derived) Feynman transform of Lie graph homology is acyclic on the one hand, and contains a subcomplex computing commutative graph homology on the other.

We thus conclude that {\it every commutative graph homology class of non-zero genus may be represented by a graph labeled by Lie graph homology classes}.  This correspondence is realized via the boundary operator in a derived version of the Feynman transform.  In particular, the differential in this acyclic complex depends not just on the modular operad structure of Lie graph homology, but also on its higher ``Massey products'' which we compute here-in.  Under this correspondence, the homology class $\sigma_{2j+1}$ corresponds to a graph whose lone non-trivial vertex label is built from $\alpha_{2j+1}$. % This is the nature of the correspondence that we make explicit for wheel graph homology classes. 

\subsection{The correspondence: technical version.}
%By Lie graph homology we refer to the homology of the Feynman transform of the $\mathfrak{K}$-twisted Lie operad  $\Sigma^{-1} s \mathsf{Lie}$; the result being an (untwisted) modular operad.

The graph complex $(\mathsf{GC}_2,d_{\mathsf{GC}})$ of \cite{WTw} which we consider here-in consists of graphs with no legs and no loops (aka tadpoles). The differential is a sum of all possible edge expansions, with signs encoded by a factor of the top exterior power of the set of edges of the graph.   This complex splits over genus, and the complex   $\mathsf{GC}_2^g$ is (after a shift in degree by $\Sigma^{-2g}$) a sub-complex of $\FT(\mathsf{Com})(g,0)$.  We will suppress the degree shift notation $\Sigma^{-2g}$ in this introduction.

The modular operad structure of Lie graph homology was studied in \cite{CHKV}.  This modular operad is not formal and so, abstractly, higher operations exist which make Lie graph homology into a modular operad up-to-homotopy (called a weak modular operad) which is homotopy equivalent to the Feynman transform of an extension by zero \cite{WardMP}.  The derived version of the Feynman transform $\FT$ mentioned above has a differential with additional terms taking these higher operations into account.   Making these higher operations explicit in all genera is both unlikely to be possible and unnecessary for the task at hand.  As we shall show, the wheel graph homology classes can be detected by higher operations on Lie graph homology landing in genus 1.  With this in mind, we proceed as follows.

First we truncate Lie graph homology above genus 1 and endow the resulting semi-classical modular operad with explicit higher operations.  The resulting weak modular operad is denoted $\wscmo$.  The projection $\wscmo\to \mathsf{Com}$ induces a map of $\mathfrak{K}$-modular operads $\FT(\mathsf{Com}) \hookrightarrow   \FT(\wscmo)$.  Declare a loop to be simple if its adjacent vertex is trivalent and of genus $0$.  The subspaces of the Feynman transform supported on graphs with simple loops form a dg submodule.  Let $\bar{\FT}$ denote passage to the quotient;  we also write $\bar{\eta}$ for the projection of a vector $\eta$ to this quotient.    Lifting $\FT(\mathsf{Com}) \hookrightarrow \FT(\wscmo)$ we have the following diagram of dg $\mathbb{S}$-modules:
	\begin{equation}\label{zz2}
\mathsf{GC}_2 \hookrightarrow \bar{\FT}(\mathsf{Com}) \hookrightarrow \bar{\FT}(\wscmo).
\end{equation}
The left hand inclusion is split; we denote the projection
$\pi\colon \bar{\FT}(\mathsf{Com}) \to \mathsf{GC}_2$.

We write $\omega_{2j+1} \in \mathsf{GC}_2^\ast$  for the wheel graph; a $2j+1$-gon along with a central vertex connected to the other vertices, tensored with a fixed element in the top exterior power of the span of the set of edges of the graph (Figure $\ref{fig:w5}$).  %In particular $\omega_{2j+1}$ has degree $4j+2$.    
The main technical result of this paper is the following construction:

\begin{theorem}\label{itsaboundary2} There exists an element  $\eta_{2j+1}\in \FT(\wscmo)(2j+1,0)$ %of degree $-4j-1$
	 such that $d(\bar{\eta}_{2j+1})\in \bar{\FT}(\mathsf{Com}) $ %\subset \FT(\wscmo)/\sim$ 
	 and $\langle \omega_{2j+1}, \pi\circ d(\bar{\eta}_{2j+1}),  \rangle \neq 0$. 
\end{theorem}  

The condition $\langle \omega_{2j+1}, \pi\circ d(\bar{\eta}_{2j+1}) \rangle \neq 0$ ensures that the elements $\pi\circ d(\bar{\eta}_{2j+1})$ are not boundaries in the submodule $\mathsf{GC}_2\subset \bar{\FT}(\wscmo)$.  On the other hand, the elements $d(\bar{\eta}_{2j+1})$ are boundaries in the larger complex and so certainly must project to cycles in $\mathsf{GC}_2$.  Thus, each $\pi\circ d(\bar{\eta}_{2j+1})$ represents a non-trivial graph homology classes in $\mathsf{GC}_2$.  Equivalently we may state:

%Since every term in the differential of $\delta(\omega_{2j+1}) \in \FT(\mathsf{Com})^\ast$ has an underlying graph with parallel edges, routine considerations show that $\omega_{2j+1}$ is a cycle.  Theorem $\ref{itsaboundary2}$ shows that it's not a boundary.  Thus:

\begin{corollary}\label{dualcor}  The homology class $[\omega_{2j+1}]\in H_0(\mathsf{GC}_2^\ast)$ is non-trivial.
\end{corollary}

The proof of Theorem $\ref{itsaboundary2}$ is an explicit construction.  The element $\eta_{2j+1}$ has underlying graph pictured in Figure $\ref{beta}$, and carries a vertex label given by the VCD class $\alpha_{2j+1}$ composed with $2j-2$ copies of the commutative product.  In this way, we find $\eta_{2j+1}$ not by contracting the exterior polygon (a logical first guess) but by contracting interior polygons (Figure $\ref{fig:w5}$).  Contracting such interior polygons produces graphs having $n=2j-2$ adjacent loops and the construction of $\eta_{2j+1}$ depends crucially on the representation theory of the Weyl group of $SO(2n+1)$ to show that the class we construct is the unique class (up to scalar multiple) which does not vanish upon passage to coinvariants by the group of isomorphisms of such a graph.

Corollary $\ref{dualcor}$ may also be derived from the results of \cite{WTw}, see \cite[Theorem 2.6]{CGPJAMS}.  What's new here is the technique, as our proof completely avoids discussion of Drinfeld associators and $\mathfrak{grt}_1$.  The interplay between Lie and commutative graph homology is subtle, and much interesting work remains.  For example how to realize Morita and Eisenstein classes via Massey products on commutative graph homology, or how to relate open conjectures on either side of this correspondence. For now we simply offer the results of this paper as a polite suggestion that commutative and Lie graph homology may be effectively studied in tandem.

%Should cite \cite{Kont2} somewhere...

This paper is organized as follows.  A brief overview of the prerequisites is given in Section 2.  In Section 3 we compute relations in the modular operad for Lie graph homology and then construct higher operations for Lie graph homology landing in genus 1.  The construction of the element $\eta_{2j+1}$ and verification of its properties is given in Section 4.  Throughout we mostly work with the ``co-Feynman transform'', and then linear dualize as the very last step (Subsection $\ref{ld}$) to prove the results stated in this introduction.

\tableofcontents

\begin{figure}
	\centering
	\includegraphics[scale=0.6]{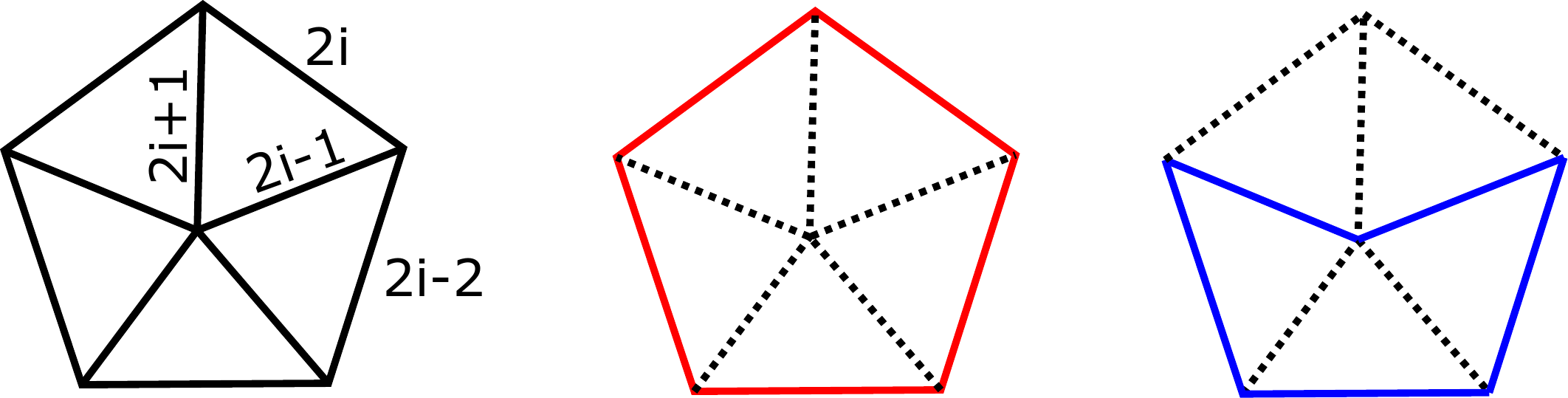}
	\caption{Left: The wheel graph $\omega_{2j+1}$ for $j=2$ with edge order indicated. Center: The ``exterior'' $2j+1$-gon is the pictured (red) subgraph.  Right:  The ``interior'' $2j+1$-gon is the pictured (blue) subgraph.}
	\label{fig:w5}
\end{figure}

\subsection{Conventions}  Throughout we work in the category of differential graded vector spaces over $\mathbb{Q}$ with homological grading conventions.  We write $A^\ast$ for the graded linear dual of $A$.  We denote symmetric groups by $S_n$ and denote irreducible representations of $S_n$ by $V_{\alpha}$ for a partition $\alpha = (\alpha_1\cdc \alpha_r)$ of $n$.

\section{Background.}

\subsection{Graphs}  We briefly review the standard operadic definition of an abstract graph, see \cite{WardMP} for full details.  A graph $\gamma = (V,F,a,\iota)$ consists of a finite non-empty set of vertices $V$, a finite set of flags (also called half-edges) $F$, a function $a\colon F\to V$ which indicates to which vertex a flag is adjacent, and an involution $\iota\colon F\to F$.  The orbits of $\iota$ of order two are called the edges of $\gamma$ and denoted $E(\gamma)$.  The orbits of order one are called the legs $\gamma$ and denoted $\text{leg}(\gamma)$.  A loop is an edge $\{f,\iota(f)\}$ for which $a(f)=a(\iota(f))$.  The number $|a^{-1}(v)|$ is called the valence of the vertex $v$.  %; a graph is called trivalent if each vertex has valence 3.  
%To any graph $\gamma$ we may associated its underlying leg-free graph, denoted $\underline{\gamma}$, by simply discarding the legs.

A graph determines a 1-dimensional CW complex and we say the graph is connected if this CW complex is connected.  A genus labeling of a graph is a function $g\colon V\to \mathbb{Z}_{\geq 0}$.  A connected, genus labeled graph is stable if $2g(v)+ |a^{-1}(v)| \geq 3$ for every $v\in V$.  A leg labeling of a graph is a bijection $\{1\cdc n\} \to \text{leg}(\gamma)$, for the appropriate $n$. A  modular graph is a stable graph along with a leg labeling.  The total genus of a modular graph is $g(\gamma):=\beta_1(\gamma)+\sum_v g(v)$, where $\beta_1$ denotes the first Betti number of the associated CW complex.  The type of a modular graph is the pair of non-negative integers $(g(\gamma), |\text{leg}(\gamma)|)$.% and is denoted $\ast_\gamma$.

An isomorphism of abstract graphs is a pair of bijections between the respective vertices and flags which commutes with the adjacency and involution maps.  An isomorphism of modular graphs is an isomorphism of abstract graphs which preserves the leg and genus labeling.  If $\gamma$ is a modular graph we write $Aut(\gamma)$ for the group of automorphisms of $\gamma$ viewed as a modular graph. % An automorphism $\phi$ of a modular graph determines a bijection on the set of edges of $\gamma$ which we denote $\bar{\phi}$.
For non-negative integers $g$ and $n$ with $2g+n\geq 3$ we fix once and for all a skeleton of the groupoid of graphs of type $(g,n)$ and call it $\mathsf{Gr}(g,n)$.

A subgraph of a graph is a pair of subsets of $V$ and $F$ closed under $a$ and $\iota$.  A nest $N$ on a graph $\gamma$ is a proper, connected subgraph of $\gamma$ containing no legs. 
% subset of the edges of $\gamma$ such that the smallest subgraph containing $N$ is connected.  
Given a nest $N$ on a graph $\gamma$ we define two auxiliary graphs.  The modular graph $\gamma/N$ is formed by contracting the edges of $N$ and their adjacent vertices to a single vertex labeled to preserve the total genus of the graph.  We call this new vertex $N$.  The graph $\hat{N}$ is the graph formed by adding as legs all flags of $F(\gamma)\setminus F(N)$ which are adjacent to vertices in $N$. 
Note the legs of $\hat{N}$ are not numerically labeled, but are in bijective correspondence with the flags adjacent to the vertex $N$ in $\gamma/N$.

Define $\mathfrak{K}^\ast(\gamma)$ to be the top exterior power of the set $E(\gamma)$.  Explicitly, $\mathfrak{K}^\ast(\gamma)$ is a 1-dimensional vector space concentrated in degree $|E(\gamma)|$ with carries an alternating action of the group $S_{E(\gamma)}$.  Define $\mathfrak{K}(\gamma)$ to be the linear dual of $\mathfrak{K}^\ast(\gamma)$.  Observe $\mathfrak{K}(\gamma)$ is naturally identified with $\mathfrak{K}^\ast(\gamma)$, but is concentrated in degree $-|E(\gamma)|$.  As we are using homological grading conventions, these definitions are opposite to the conventions in \cite{GeK2}.  A mod 2 order on the edges of a graph is defined to be a choice of unit vector in $\mathfrak{K}^\ast(\gamma)$ or equivalently $\mathfrak{K}(\gamma)$.

\begin{figure}
	\centering
	\includegraphics[scale=1.1]{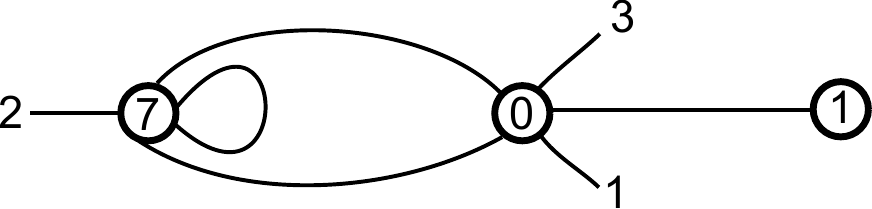}
	\caption{A modular graph of type $(10,3)$ having $3$ vertices, $11$ flags, $4$ edges, $1$ loop and $3$ legs.}
	\label{fig:graph}
\end{figure}

\subsection{Weak modular operads}
A stable $\mathbb{S}$-module $A$ is a family of differential graded $S_n$ representations $A(g,n)$, indexed over pairs of non-negative integers $(g,n)$ satisfying $2g+n\geq 3$.  Given such an $A$ we extend $A(g,-)$ to a functor valued in all finite sets via left Kan extension.  In particular, if $X$ is a finite set 
$A(g,X)$ is non-canonically isomorphic to $A(g,|X|)$.  Given a stable $\mathbb{S}$-module $A$ and a modular graph $\gamma$ with vertex $v$ we define $A(v):= A(g(v), a^{-1}(v))$ and define $A(\gamma):=\tensor_{v\in V(\gamma) } A(v)$.  Note $A(\gamma)$ inherits an action of the group $Aut(\gamma)$.

\begin{definition}\label{wmodef} Let $A$ be a stable $\mathbb{S}$-module.  A weak modular operad structure on $A$ is a collection of degree $-1$ operations
\begin{equation}\label{operations} \mathfrak{K}(\gamma)^\ast \tensor A(\gamma)\stackrel{\mu_\gamma}\to A(g,n),
\end{equation}
for each $\gamma \in \mathsf{Gr}(g,n)$ and for all $(g,n)$ which
are $S_n$-equivariant, $Aut(\gamma)$-coinvariant and which 
 satisfy the differential condition 
$ \sum  \mu_{\gamma/N}\circ_N \mu_{\hat{N}}=0 $.
\end{definition}
Here $\mu_{\gamma/N}\circ_N \mu_{\hat{N}}$ is the composition of operations given by plugging the output of $\mu_{\hat{N}}$ into the vertex $N$ of $\gamma/N$ and the sum is over all nests $N$ on $\gamma$.  The factors of $\mathfrak{K}$ are composed by pulling back the wedge product isomorphism:
\begin{equation*}
\mathfrak{K}(\gamma/N)\tensor \mathfrak{K}(\hat{N})\stackrel{\wedge}\to \mathfrak{K}(\gamma),
\end{equation*}
which in turn encodes signs in the differential.  We refer to \cite[Proposition 3.21]{WardMP} for additional details.  %We also remark that given a weak modular operad there is a natural extension of the operations $\mu_\gamma$ to graphs whose leaves are labeled by an arbitrary finite set, which we will regularly make use of below.
  Classical modular operads are weak modular operads for which $\mu_\gamma=0$ if $|E(\gamma)|>1$.  In this case, the differential condition collapses to the associativity of composition of these one edged operations.

\subsection{Co-Feynman transform}\label{coftsec}  Let $A$ be a weak modular operad.  %For simplicity, assume the internal differential $d_A=0$.  
The co-Feynman transform of $A$, denoted $\mathsf{B}(A)$ is defined to be its bar construction viewed as an algebra over the Koszul resolution of the groupoid colored operad encoding modular operads \cite{WardMP}.  Explicitly this means
$\mathsf{B}(A)$ is the dg $\mathfrak{K}$-modular co-operad given by:
\begin{equation*}
\mathsf{B}(A)(g,n) = \ds\bigoplus_{\gamma \in \mathsf{Gr}(g,n)} \mathfrak{K}^\ast(\gamma)  \tensor_{Aut(\gamma)} A(\gamma).
\end{equation*}
The $\mathfrak{K}$-modular co-operad structure is co-free and specified by decomposition maps
\begin{equation*}
\nu_\gamma \colon \mathsf{B}(A)(g,n) \to \mathfrak{K}^\ast(\gamma)\tensor_{Aut(\gamma)}\mathsf{B}(A)(\gamma)
\end{equation*}
for each $\gamma\in \mathsf{Gr}(g,n)$. These decomposition maps are defined on the summand of the source corresponding to a graph $\gamma^\prime$ by summing over all ways to add a single layer of nests
 %\footnote{This means every vertex is contained with in exactly one nest.  This could also be called a depth 1 nesting.} 
to it such that $\gamma^\prime/\sqcup N_i=\gamma$. %\footnote{Along with the sign convention that outer edges are placed in the last position.}  %The point is that the target corresponds to nested graphs whose outer graph is $\gamma$.  Note if two nestings are related by an automorphism we take both.

The weak $\mathfrak{K}$-modular operad structure maps on $A$ induce
a degree $-1$ map of stable $\mathbb{S}$-modules $\mu\colon \mathsf{B}(A) \to A$.  This map induces a  differential $\partial\colon \mathsf{B}(A)(g,n)\to \mathsf{B}(A)(g,n)$.  To describe $\partial$ it is sufficient to indicate its composite with projection $\pi_\gamma$ to a summand indexed by a $(g,n)$-graph $\gamma$, which is defined by asserting that the following diagram commutes:
\begin{equation}\label{coftd}
\xymatrix{  \mathsf{B}(A)(g,n) \ar[d]^{\nu_\gamma} \ar[r]^{ \ \ \ \ \partial} & \mathsf{B}(A)(g,n) \ar@{->>}[d]^{\pi_\gamma}  \\ 
	\mathsf{B}(A)(\gamma) \tensor_{Aut(\gamma)} \mathfrak{K}^\ast(\gamma) \ar[r]^{\ \ \ \overline{\mu}}	&  A(\gamma) \tensor_{Aut(\gamma)}\mathfrak{K}^\ast(\gamma).
}
\end{equation}	
Here $\overline{\mu}$ is the extension of $\mu$ by the Leibniz rule, composed with $-\tensor_{Aut(\gamma)}\mathfrak{K}^\ast(\gamma)$.  In particular, $\overline{\mu}$ is supported on nested graphs in $\mathsf{B}(A)(\gamma)$ for which exactly one vertex of $\gamma$ is not labeled by a corolla.%\footnote{The signs can be described by ``first out -- last in'' meaning we apply the decomposition after having permuted the nested edges to be in the last position and the outer edges to be in the first position.}

%\footnote{Questions of finiteness, square-zero, compatibility with co-operad structure all follow from the general theory which can be referenced here.}

The co-Feynman transform of a weak modular operad has a bigrading
\begin{equation}\label{bigrading}
\mathsf{B}(A)(g,n) = \ds\bigoplus_{r\geq 0, s\in \mathbb{Z}} \mathsf{B}(A)(g,n)^{r,s}
\end{equation}
given by the number of edges $r$ on the graph indexing a summand and $s$ is the sum of the internal degrees of the vertex labels.  With respect to this bigrading, the co-Feynman transform differential is supported on:
\begin{equation}\label{bigrading2}
\mathsf{B}(A)(g,n)^{r,s} \stackrel{\partial}\to \ds\bigoplus_{0\leq e \leq r} \mathsf{B}(A)(g,n)^{r-e, s+e-1}.
\end{equation}
In particular, on each bigraded component the differential splits as $\oplus_e\partial_e$ corresponding to those terms which contract $e$ edges, under convention that contracting $0$ edges means apply the internal differential $d_A$. 

\begin{definition}  When each $A(g,n)$ is finite dimensional in each graded component we define the (weak) Feynman transform $\FT$ to be the linear dual of the co-Feynman transform.  In particular, the Feynman transform of a weak modular operad is a $\mathfrak{K}$-modular operad.
\end{definition}

This definition generalizes the usual Feynman transform \cite{GeK2}, under the definition that a (strict) modular operad is a weak modular operad for which $\mu_\gamma =0$ whenever $\gamma$ has more than one edge.  %In general however, this weak Feynman transform has differential terms corresponding to the higher order compositions of $(A,\mu)$.

\begin{remark}\label{d1rmk} Let $(A, \mu)$ be a weak modular operad with vanishing internal differentials $d_A=0$.  Then $A$ along with only its one-edged contractions forms a (strict) modular operad.  In this case $\partial_1$ is itself square zero; it is the linear dual of the differential in the usual Feynman transform of this (strict) modular operad.
\end{remark}

\subsection{$\FT(\mathsf{Com})$ and $\mathsf{GC_2}$} \label{com}
As above, $\mathsf{Com}$ is the operad encoding commutative algebras, viewed as a modular operad by extension to higher genus by $0$.  %Thus we consider $\mathsf{Com}(g,n)$ to be the ground field if $g=0$ and $0$ otherwise.  
It follows that
\begin{equation*}
\FT(\mathsf{Com})(g,n)\cong \ds\bigoplus \mathfrak{K}(\gamma)_{Aut(\gamma)},
\end{equation*}
where the sum is taken over all $\gamma \in \mathsf{Gr}(g,n)$ for which $g(v)=0$ for all $v\in V(\gamma)$. 

Since all vertices have genus $0$, no differential terms expand loops, and so the sum over those graphs with no loops is a subcomplex $\mathsf{NL}(g,n)\subset \FT(\mathsf{Com})(g,n)$.  For each $g\geq 3$ we define the chain complex $\mathsf{GC}_2^g = (\Sigma^{2g} \mathsf{NL} (g,0))$. We then define $\mathsf{GC}_2 = \prod_g \mathsf{GC}_2^g$.  Note that \cite{WTw} uses cohomological conventions, so to recover exactly his $\mathsf{GC}_2$, one must take the cochain complex associated to the chain complex which we have called $\mathsf{GC}_2$ by negating the indices: $(V_i)^{op}= V_{-i}$.

In particular, a homogeneous element in $\mathsf{GC}_2$ is specified by a scalar multiple of an isomorphism class of a connected graph with no loops, all of whose vertices have valence 3 or greater, along with a mod 2 order on the set of edges. The degree of such a vector is $2g-E=E-2V+2$ and the differential is given by a sum of edge expansions.  The alternating action on the edges of a representative has the effect that any isomorphism class of a graph with parallel edges  vanishes.

\subsection{Wheel graphs}
As above we define the wheel graph $\omega_{2j+1}$ by connecting a new vertex to all the edges in a $2j+1$-gon.  As a convention the mod 2 edge order is chosen to coincide with \cite[Proposition 9.1]{WTw}, see Figure $\ref{fig:w5}$.  

With this convention the wheel graph $\omega_{2j+1}$ may be viewed as a degree $4j+2$ element of $\mathsf{B}(\mathsf{Com})(2j+1,0)$, having each vertex labeled by the commutative product of suitable valence.  Note that contracting any edge in a wheel graph produces parallel edges with commutative labels and hence:
\begin{lemma}\label{cycle} The wheel graph $\omega_{2j+1}$ is a cycle in the complex $\mathsf{B}(\mathsf{Com})(2j+1,0)$.
\end{lemma}
%If we view $\omega_{2j+1}\in\mathsf{B}(\mathsf{Com})(2j+1,0)$ it is no longer true that $\partial(\omega_{2j+1})=0$, however the lemma does imply $\partial_1(\omega_{2j+1})=0$.

We remark that the vector spaces $\mathsf{B}(\mathsf{Com})(2j+1,0)$ and $ \FT(\mathsf{Com})(2j+1,0)$ are canonically isomorphic, via $\mathsf{Com}(g,n)\cong \mathsf{Com}(g,n)^\ast$, and so the wheel graph with the above conventions also specifies an element $\omega_{2j+1}^\ast\in \FT(\mathsf{Com})(2j+1,0)$.  Lemma $\ref{cycle}$ shows that a vector for which $\omega_{2j+1}^\ast$ appears with non-zero coefficient is not a boundary in $\FT(\mathsf{Com})(2j+1,0)$.

\section{Higher operations on Lie graph homology.}

\subsection{Lie graph homology}

Following \cite{GeK2}, after \cite{Kont2}, 
%A formal definition of what we call Lie graph homology may be given in terms of the Feynman transform.  
we define Lie graph homology to be the modular operad
$H_\ast(\FT(\Sigma s^{-1} \mathsf{Lie}))$, where $s$ denotes the (cyclic)  operadic suspension and $\Sigma$ denotes an shift up in degree.  In particular extension by zero makes $\Sigma s^{-1} \mathsf{Lie}$ a $\mathfrak{K}$-modular operad, and its Feynman transform is a modular operad.
Following \cite{CHKV} we will denote this modular operad by $H_\ast(\Gamma)$ and use the notation $H_\ast(\Gamma_{g,n})= H_\ast(\Gamma)(g,n)$.

For our purposes however, we will only require the following partial characterization of this modular operad in genus $\leq 1$. 

\begin{lemma}\label{gamma}\cite{CHKV}  The graded vector spaces $H_\ast(\Gamma_{g,n})$ form a modular operad with the following properties:
\begin{enumerate}
	\item The underlying cyclic operad $H_\ast(\Gamma_{0,-})$ is canonically isomorphic to the commutative (cyclic) operad.

\item As an $S_n$ module, 
\begin{equation*} H_i(\Gamma_{1,n}) \cong \begin{cases} V_{n-i,1^i} & \text{ if $i$ is even and } 0\leq i \leq n-1 \\
0 & \text{else} \end{cases}
\end{equation*} 

\item The modular operadic composition map
\begin{equation*}
H_0(\Gamma_{0,3})\tensor H_{d}(\Gamma_{1,n}) \to H_{d}(\Gamma_{1,n+1})
\end{equation*}
is injective.
\end{enumerate}
\end{lemma}

Notice that $H_{2j}(\Gamma_{1,2j+1})$ is the alternating representation.  We fix generators $\alpha_{2j+1} \in H_{2j}(\Gamma_{1,2j+1})$ once and for all. We will also write $m_t\in \mathsf{Com}(t)= H_0(\Gamma_{0,t+1})$ for the commutative product.

We may iterate the modular operadic composition map above to form
\begin{equation}\label{circe}
H_0(\Gamma_{0,t+1})\tensor H_{d}(\Gamma_{1,n}) \stackrel{\circ_e}\longrightarrow H_{d}(\Gamma_{1,n+t-1}).
\end{equation}
Here $\circ_e$ is the modular operadic composition which corresponds to gluing along a tree with one edge adjacent to vertices of type $(0,t+1)$ and $(1,n)$.  To be precise, this composition is only well defined after a choice of labeling of the $t+n-1$ legs of the tree by the set $\{1\cdc t+n-1\}$.  We fix the convention that $\circ_e$ corresponds to labeling the genus $0$ vertex by $\{1\cdc t\}$ and the genus 1 vertex by $\{t+1\cdc t+n-1\}$.  Observe that repeated application of Lemma $\ref{gamma} (3)$ implies that $m_t\circ_e\alpha_{2j+1}\neq 0$.

 We now calculate relations between compositions in the modular operad $H_\ast(\Gamma)$.
\begin{lemma} \label{relations} The non-zero homology class $m_t\circ_e \alpha_{2j+1} \in H_{2j}(\Gamma_{1,t+2j})$ satisfies 
\begin{equation*}
\sum_{i=1}^{t+1}(i,t+1) (m_t\circ_e \alpha_{2j+1}) =0,
\end{equation*}	
where $(i,t+1)\in S_{t+2j}$ denotes a transposition. 
\end{lemma}
\begin{proof}
The map $\circ_e$ of Equation $\ref{circe}$ is $S_{t}\times S_{2j}$ equivariant, where the target carries the restricted action along the standard inclusion $S_{t}\times S_{2j}\hookrightarrow S_{t+2j}$.   By abuse of notation we also write $\circ_e$ for the adjoint:
\begin{equation}\label{adjointform}
 Ind^{S_{t+2j}}_{S_{t}\times S_{2j} } (H_0(\Gamma_{0,t+1})\tensor H_{2j}(\Gamma_{1,2j+1})) \stackrel{\circ_e}\to H_{2j}(\Gamma_{1,2j+t}).
\end{equation}
Using the Littlewood-Richardson rule (\cite[p.456]{FH}) we compute the irreducible decomposition of the source of Equation $\ref{adjointform}$ to be $V_{t+1,1^{2j-1}}\oplus V_{t,1^{2j}}$.

By Lemma $\ref{gamma}$, the target of $\circ_e$ is an irreducible $S_{t+2j}$-representation of type $V_{t, 1^{2j}}$.  Thus, any vector $z$ for which $\mathbb{Q}[S_{t+2j}]\cdot z\cong V_{t+1,1^{2j-1}}$ must be in the kernel of $\circ_e$.  To produce such a vector $z$ we embed the problem in the group ring.  That is, consider the $S_t\times S_{2j}$ equivariant map $H_0(\Gamma_{0,t+1})\tensor H_{2j}(\Gamma_{1,2j+1})  \to  \mathbb{Q}[S_{t+2j}]$ defined by 
\begin{equation}\label{y}
 m_t\tensor \alpha_{2j+1}  \mapsto  y:= (\sum_{\sigma \in S_{t}} \sigma)( \sum_{\sigma \in S_{\{t+1,\dots t+2j\}}} sgn(\sigma) \sigma).
\end{equation}
Form the Young diagram of shape $t+1,1^{2j-1}$ labeled numerically right to left, then down.  So $t+1$ is in the pivot position.  Call this tableau $\zeta$.  Its associated Young symmetrizer $c_\zeta$ is
\begin{equation*}
c_\zeta = (id+(1,t+1)+(2,t+1)+\dots+(t,t+1))y.
\end{equation*}
Since $y$ is in the image of the map defined in Equation $\ref{y}$, $c_\zeta$ is in the image of the adjoint morphism from the induced representation.  By construction $c_\zeta$ generates a copy of $V_{t+1,1^{2j-1}}$ under left multiplication by $\mathbb{Q}[S_{t+2j}]$, hence so does
\begin{equation*}
z:=(id+(1,t+1)+(2,t+1)+\dots+(t,t+1)) (m_t\tensor \alpha_{2j+1}). 
\end{equation*}
Thus, this $z$ is in the kernel of the $\circ_e$ in Equation $\ref{adjointform}$ as desired.   Moreover the kernel is spanned by the $S_{t+2j}$ orbit of this $z$.
\end{proof}

\subsection{The weak semi-classical modular operad $\wscmo$}

In this section we endow the genus $\leq 1$ spaces of $H_\ast(\Gamma_{g,n})$
with higher operations.  The result will be a weak modular operad which we denote $(\wscmo, \mu)$.
 
As stable $\mathbb{S}$-modules we define
\begin{equation*}
\wscmo(g,n) = \begin{cases}
H_\ast(\Gamma_{g,n}) & \text{ if } g < 2 \\
0 & \text{ if } g\geq 2 
\end{cases}
\end{equation*}

The operations $\mu_\gamma$ are defined as follows. % First $\mu(\gamma)$ is necessarily $0$ if $g(\gamma)\geq 2$.  
We define $\mu_\gamma=0$ unless one of the two mutually exclusive conditions is met:
\begin{itemize}
	\item $\gamma$ has genus $< 2$ and only one edge, or
	\item $\gamma$ has genus 1, and the underlying leg free graph of $\gamma$ is a $2j+1$-gon, for some $j\geq 1$.
\end{itemize}

In the first case we define  $\mu_\gamma$ to be the operation induced by the modular operad structure on $H_\ast(\Gamma)$.  In the second case we proceed as follows.

Let $\mathsf{p}_{2j+1}$ be the standard trivalent $2j+1$-gon.  By this we mean the modular graph formed by attaching a leg to each vertex of a $2j+1$-gon.  The vertices have genus label $0$.  The leg labels are in the dihedral order, and we give this modular graph an edge ordering $e_1<\dots <e_{2j+1}$ such that the $i^{th}$ edge connects the vertices adjacent to flags $i$ and $i+1$ (mod $2j+1$).

Observe that $\wscmo(\mathsf{p}_{2j+1})\cong k$.  We define
\begin{equation*}
\mu_{\mathsf{p}_{2j+1}} \colon \wscmo(\mathsf{p}_{2j+1})\tensor \mathfrak{K}(\mathsf{p}_{2j+1}) \to \wscmo(1, 2j+1)
\end{equation*}
by  $\mu_{\mathsf{p}_{2j+1}}(1\tensor e_1\wedge \dots \wedge e_{2j+1})=\alpha_{2j+1}$.

\begin{lemma} \label{H} The above operations extend to a unique weak modular operad structure on $\wscmo$.
\end{lemma}
\begin{proof}  
	We refer to Definition $\ref{wmodef}$.  The $S_n$ equivariance defines $\mu_{\hat{\mathsf{p}}}$ for any other edge ordered trivalent polygon $\hat{\mathsf{p}}$. One easily checks that this definition is not over-prescribed, since symmetries of a $2j+1$-gon induce permutations of the edges and the legs which have matching parity.  %Thus applying such a symmetry produces the sign of the given permutation twice (once from the factor of $\mathfrak{K}(\mathsf{p})$ and once since the target is alternating), and hence has no effect.

We then want to show that if $\mathsf{P}$ is a non-trivalent graph whose underlying leg free graph is a $2j+1$-gon, that $\mu_{\mathsf{P}}$ is determined by the above operations.  For this we induct on the number of non-trivalent vertices.  First suppose that this number is 1, at a vertex $v$ of $\mathsf{P}_1:=\mathsf{P}$.

Let $\gamma$ be the graph formed by blowing up $v$ to separate the two flags which belong to edges of the polygonal subgraph from the rest of the flags at $v$; see Figure $\ref{fig:polygons}$.  In particular $\gamma$ has $2j+2$ edges, $2j+1$ of which form a polygon $\mathsf{P}_0$ with trivalent vertices.  Let $N$ be a nest on $\gamma$.  By the above definition, the composition $\mu_{\gamma/N}\circ_N \mu_{\hat{N}}$ will be zero unless $N = \{e\}$ or $\hat{N}= \mathsf{P}_0$ (pictured blue and red in Figure $\ref{fig:polygons}$).

Thus, applying the differential condition of Definition $\ref{wmodef}$ with internal differential $d=0$, it must be the case that  %\footnote{	There is a sign check here, but since you define the new operation by collapsing a single edge in the last position, the sign is independent of the order.  Actually, the sign can be determined explicitly because the kappa part is +, and the gamma/N part is 2j+1, so you get +=+.}
	\begin{equation}\label{d}
0=\mu_{{\mathsf{P}_1}} \circ \mu_{e} + \mu_{\gamma/{\mathsf{P}_0}} \circ \mu_{\mathsf{P}_0}.
	\end{equation}	
Here we write $\circ_e= \mu_e$ for the modular operadic composition map which contracts the edge $e$.  But this map $\mu_e\colon\wscmo(\gamma)\to \wscmo(\mathsf{P})$ is simply a composition in the commutative operad and so is an isomorphism, $\wscmo(\gamma)\stackrel{\cong}\to \wscmo(\mathsf{P})$.  Thus Equation $\ref{d}$ uniquely determines $\mu_{\mathsf{P}_1}$.

\begin{figure}
	\centering
	\includegraphics[scale=.8]{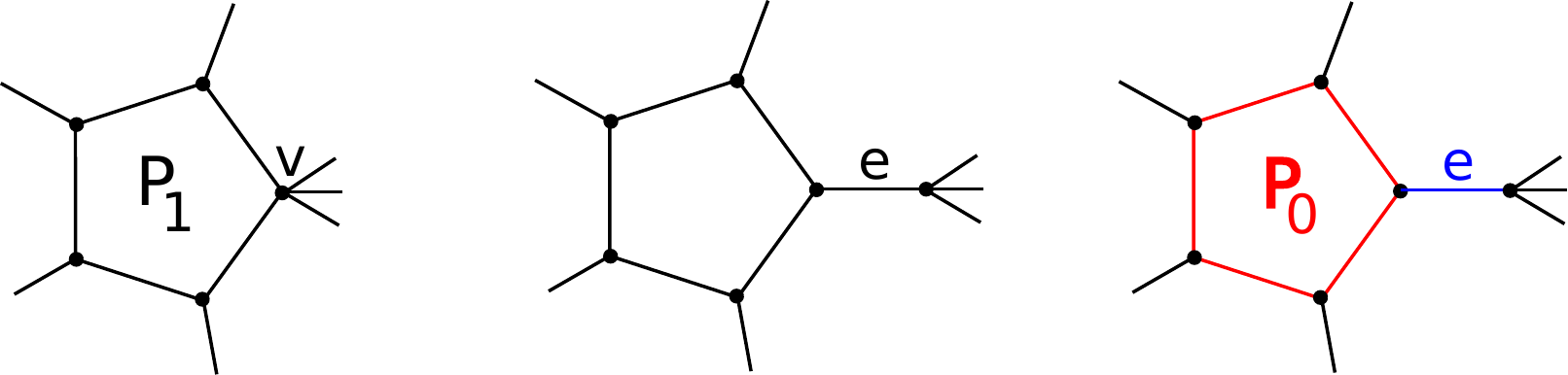}
	\caption{The Massey project associated to $\mathsf{P}=\mathsf{P}_1$ is defined by expanding the edge $e$, applying the Massey product $\mu_{\mathsf{P}_0}$ and then the modular operadic composition $\circ_e$.  Leg labels are suppressed in the figure.}
	\label{fig:polygons}
\end{figure}

For the induction step we repeat the above argument, reducing the number of non-trivalent vertices by one at each step.  The fact that the operation is independent of the choice of order of the non-trivalent vertex follows from the fact that $H_\ast(\Gamma)$ is a (strong) modular operad.

Thus there is at most one weak modular structure on $\wscmo$ extending the above operations.  Conversely the maps defined above have the requisite degree and equivariance, so it remains to show that the differential condition
\begin{equation}\label{dcond}
\ds\sum_{N \text{ on } \gamma}  \mu_{\gamma/N}\circ_N \mu_{\hat{N}}=0 
\end{equation}
is satisfied for every modular graph $\gamma$.

If $\gamma$ has total genus $\geq 2$ or has fewer than $2$ edges, all terms in Equation $\ref{dcond}$ are $0$ by definition so there is nothing to check.
If $\gamma$ has exactly two edges, then Equation $\ref{dcond}$ is merely the associativity axiom for the (strict) modular operad structure on $H_\ast(\Gamma_{g,n})$.

So we now assume $\gamma$ has more than two edges and has total genus $\leq 1$.  Let $N$ be a nest on $\gamma$.  Either $N$ or $\gamma/N$ has two or more edges, thus for $\mu(\gamma/N)\circ_N\mu(\hat{N})$ to be non-zero requires that one of $\gamma/N$ and $N$ is a polygon with an odd number of sides and the other must be a lone edge.  In particular $\gamma$ must have an even number of edges, have first betti number 1, and hence only genus 0 vertices.  There are two cases for such a $\gamma$.  Either its lone cycle has an odd number of edges, in which case there is an additional edge pointing outward or its lone cycle has an even number of edges, in which case the leg free graph underlying $\gamma$ must be a $2n$-gon.

The first case follows as above.  In particular, suppose $\gamma$ has $2n$ edges and its lone cycle is of length $2n-1$.  Let $e$ be the unique edge of $\gamma$ which is not in the cycle. This edge is connected to vertices $v$ and $w$, with $v$ belonging to the polygon and $w$ not.  If $v$ is a trivalent vertex in $\gamma$, then the differential condition was verified above.   If $v$ is not trivalent, the differential condition is verified by a double iteration of Equation $\ref{d}$.

So we now consider the second case.  Suppose $\gamma = \mathsf{P}_{t}$ is a polygon with, $2n \geq 4$ sides whose vertices have genus 0, and which has $t\geq 2n$ legs labeled $\{1\cdc t\}$, with at least one leg at each vertex. The only non-zero terms in the differential condition are given by first choosing $N$ to be a single edge.
% or $\gamma$, so that $\gamma/N =  \mathsf{P}_{2n}/e$.  
Thus the differential condition in this case is:
$ 0 = \sum_{e\in \mathsf{P}_{t}} \mu_{\mathsf{P}_{t}/e}\circ \mu_{e}.
$
This condition can be rephrased as saying the following composite vanishes:
\begin{equation}\label{dsquared}
\xymatrix{ \mathsf{B}(\wscmo)(1, t)^{2n, 0} \supset
	\wscmo(\mathsf{P}_{t})\tensor \mathfrak{K}^\ast(\mathsf{P}_{t})  \ar[d]^{\partial_1}  && \\  	\mathsf{B}(\wscmo)(1, t)^{2n-1, 0} \supset
	\ds\bigoplus_{e \in \mathsf{P}_{t}} \wscmo(\mathsf{P}_{t}/e) \ar[rr]^{ \ \ \ \ \ \ \ \ \ \ \ \ \ \  \ \ \ \ \ \ \ \ \ \ \ \oplus_{e} \mu_{\mathsf{P}_{t}/e}}\tensor\mathfrak{K}^{\ast}(\mathsf{P}_{t}/e) 
  &&  H_{2n-2}(\Gamma_{1,t}) 
}
\end{equation}
Let us first consider the case $t=2n$.  Since the maps in Diagram $\ref{dsquared}$ are $S_{2n}$-equivariant it is sufficient to consider the case that the legs of $\mathsf{P}_{2n}$ are labeled in the dihedral order.  The source of this diagram is closed under the $\mathbb{Z}_{2n}$ action restricted along the standard inclusion $\mathbb{Z}_{2n} \cong \langle \sigma:=(12 \dots 2n)\rangle \subset S_{2n}$.
 Denoting this $1$-dimensional $\mathbb{Z}_{2n}$ representation $\op{E}:=Res^{S_{2n}}_{\mathbb{Z}_{2n}}(\wscmo(\mathsf{P}_{2n})\tensor \mathfrak{K}^\ast(\mathsf{P}_{2n}))$, we compute its character $\chi_{\op{E}}(\sigma)=sgn(\sigma)=-1$.  This in turn determines the isomorphism type of the irreducible $\mathbb{Z}_{2n}$-representation $\op{E}$; namely the $2n$ cycle $\sigma$ acts by $-1$.

To show the composite in Diagram $\ref{dsquared}$ is $0$, it thus suffices to show that there are no copies of this irreducible representation appearing in $Res_{\mathbb{Z}_{2n}}^{S_{2n}}(H_{2n-2}(\Gamma_{1,2n}))$.  The irreducible representations of $\mathbb{Z}_{2n}$ over the algebraic closure $\overline{\mathbb{Q}}$ are all $1$-dimensional and are given by letting $\sigma$ act by multiplication of a root of $x^{2n}-1$.  Let $\omega = e^{i\pi/n}$ and write $W_i$ for the irreducible representation corresponding to multiplication by $\omega^i$.   Let $V_{2n}\oplus V_{2n-1,1}$ be the permutation representation of $S_{2n}$.  One easily calculates its restriction 
\begin{equation*}
Res_{\mathbb{Z}_{2n}}^{S_{2n}}(V_{2n}\oplus V_{2n-1,1}) = %\bigoplus_{i=0}^{2n-1} a_i W_i = 
\bigoplus_{i=0}^{2n-1} W_i
\end{equation*}
%by imposing the requirements that the non-negative integers $a_i$ satisfy
%\begin{equation*}
%0=\chi_{(V_{2n}\oplus V_{2n-1,1})}(\sigma)= \sum a_i \omega^i \ \ \ \ \text{ and } \ \ \ \  2n=\chi_{(V_{2n}\oplus V_{2n-1,1})}(id)= \sum a_i.
%\end{equation*}
%Hence each $a_i=1$.  
% follows since each omega^j is a root of sum a_ix^i except j=0
Since $V_{2n-1,1}\tensor V_{1^{2n}} \cong V_{2,1^{2n-2}}$, the number of copies of $\op{E} \cong W_n$ appearing in $Res_{\mathbb{Z}_{2n}}^{S_{2n}}(V_{2,1^{2n-2}})$ is the number copies of $W_0$ appearing in $Res_{\mathbb{Z}_{2n}}^{S_{2n}}(V_{2n-1,1})$ which is  $1-1 = 0$, as desired.

Whence the case $t=2n$.  Now suppose $t>2n$.  Choose an ordering of the vertices of $\mathsf{P}$ compatible with the dihedral ordering and let $T_i$ be the set of flags adjacent to the $i^{th}$ vertex.  In particular $\{1\cdc t\} = \sqcup T_i$.  Let $\hat{T}_i$ be the set $T_i$ along with an added basepoint, called the root.

Consider the following diagram:
\begin{equation*}
\xymatrix{ 
 \ar[r]^{\ \ \ \ \ \ \ \ \cong} 
\left(\wscmo(\mathsf{P}_{2n})\tensor \mathfrak{K}^\ast(\mathsf{P}_{2n})\right) 
\tensor\left( \ds\bigotimes_{i=1}^{2n} \wscmo(0,\hat{T}_i)\right)  \ar[d]^{\partial_1\tensor id} 
 & 
 \wscmo(\mathsf{P}_t)\tensor \mathfrak{K}^\ast(\mathsf{P}_{t}) \ar[d]^{\partial_1}	 
	\\ 
		 \ar[r]^{\ \ \ \ \ \ \ \ \ \ \ \ \cong}
	\ds\bigoplus_{e \in \mathsf{P}_{2n}} \wscmo(\mathsf{P}_{2n}/e)  \tensor\mathfrak{K}^{\ast}(\mathsf{P}_{2n}/e) \tensor\left( \bigotimes_{i=1}^{2n} \wscmo(0,\hat{T}_i)\right)   \ar[d]^{ \oplus_{e} \mu_{\mathsf{P}_{2n}/e} \tensor id } 
	 & 
	 \bigoplus_{e\in \mathsf{P}_t}  \wscmo(\mathsf{P}_t/e)\tensor \mathfrak{K}^\ast(\mathsf{P}_{t}/e)
	  \ar[d]^{ \oplus_{e} \mu_{\mathsf{P}_{t}/e}}
	\\
	  \ar[r]^{\ \ \ \ {\sum\circ_i}}
H_{2n-2}(\Gamma_{1,2n}) \tensor (\tensor_i H_{0}(\Gamma_{0,\hat{T}_i}))
	  & 
		 H_{2n-2}(\Gamma_{1,t}) 
}
\end{equation*}
The right hand side of this diagram is exactly Diagram $\ref{dsquared}$.  The left hand side of this diagram is Diagram $\ref{dsquared}$ in the prior case $t=2n$, then tensored with the 1-dimensional, trivial $\mathbb{Z}_{2n}$-representation $\bigotimes_{i=1}^{2n} \wscmo(0,\hat{T}_i)$.  The horizontal arrows are contractions using the modular operad structure along the graph identifying the root of $\hat{T_i}$ with leg $i$ of $\mathsf{P}_{2n}$.

The commutativity of the top square can be seen just by looking at each summand -- both routes give the same graph with commutative labels.  The commutativity of the bottom square follows immediately from the definition of the Massey product associated to each ${{\mathsf{P}_t}/e}$. 

Since the left hand side of the diagram vanishes, by the $t=2n$ case considered above, and since the top horizontal arrow is an isomorphism, the right hand side of the diagram also vanishes, as desired.
\end{proof}

Viewing $\mathsf{Com}$ as a modular operad, as in Subsection $\ref{com}$, we see immediately that there is a  level-wise surjective morphism of weak modular operads $\wscmo\to \mathsf{Com}$.  Taking the weak Feynman transform, we have a level-wise injective morphism of $\mathfrak{K}$-twisted modular operads $\FT(\mathsf{Com})\hookrightarrow \FT(\wscmo)$.

\section{Proof of the main results.}
Let us regard the wreath product $S_2\wr S_n$ as follows.  Its underlying set is $S_2^n\times S_n$.  To an element $((\alpha_1\cdc \alpha_n), \tau)$ in this set we associate a permutation in $S_{2n}$ by first acting by $\tau$ on the ordered set $\{1,2\},\{3,4\} \cdc \{2n-1 , 2n\}$ of size $n$ and then acting by $\alpha_i$ on the ordered set $\{2\tau(i)-1,2\tau(i)\}$ for each $i$. This defines an injective map of sets
\begin{equation*}
S_2^{\times n}\times S_n\hookrightarrow S_{2n},
\end{equation*}
and $S_2\wr S_n$ carries the unique group structure for which this map is a homomorphism.  The wreath product $S_2\wr S_n$ has a 1-dimensional representation given by letting an element $((\alpha_1\cdc \alpha_n), \tau)$ act by multiplication by $sgn(\tau)$.  We call this representation $L_n$ ($L$ stands for loops).

In what follows, we abuse notation by regarding the sequence of injections
\begin{equation*}
(S_2)^{r-1}\hookrightarrow   S_2\wr S_{r-1} \hookrightarrow S_2\wr S_{r-1} \times S_{n-r+2} \hookrightarrow 
S_{2r-2} \times S_{n-r+2}\hookrightarrow S_{n+r}
\end{equation*}
as a sequence of subgroups. 	We write $Res_H^G$ for the restriction of a representation of a group $G$ to a representation of a subgroup $H$.

\begin{lemma} \label{replemma} Let $2\leq r < n$ be integers.  % The irreducible decomposition of the $(S_2\wr S_{r-1}) \times S_{n-r+2}$ representation  $Res_{(S_2\wr S_{r-1})  \times S_{n-r+2}}^{S_{n+r}}(V_{r,1^n})$ contains:
	\begin{itemize}
		\item  The irreducible decomposition of %the $(S_2\wr S_{r-1}) \times S_{n-r+2}$ representation  
		$Res_{(S_2\wr S_{r-1})  \times S_{n-r+2}}^{S_{n+r}}(V_{r,1^n})$ contains a unique summand of the form $L_{r-1}\boxtimes V_\beta$.  It is of the form $L_{r-1}\boxtimes V_{1^{n-r+2}}$.
		
		\item The irreducible decomposition of $Res_{(S_2\wr S_{q})  \times S_{n-2q+r}}^{S_{n+r}}(V_{r,1^n})$ has no summand of the form $L_{q}\boxtimes - $ for $q\geq r$. 
	\end{itemize} 
\end{lemma}
\begin{proof}  
	
Fix $q\geq r-1$.  The number of copies of a summand $V_{\alpha}\boxtimes V_{\beta}$ appearing in 
\begin{equation} \label{res}
Res_{S_{2q}\times S_{n+r-2q}}^{S_{n+r}}(V_{r,1^n})
\end{equation}
is computed via the Littlewood-Richardson rule \cite{FH}.  In this case, because $V_{r,1^n}$ is a hook, each $\alpha$ and $\beta$ appearing with non-zero coefficient must also be hooks.   If $q>r-1$, the number of summands of the from $V_{q+1,1^{q-1}}\boxtimes V_\beta$ appearing in the decomposition of the representation in Equation $\ref{res}$ is zero. % (since $q+1>r$).  
If $q=r-1$, there is a unique summand of the from $V_{q+1,1^{q-1}}\boxtimes V_\beta$ appearing in the decomposition of the representation in Equation $\ref{res}$ is zero.  It is of the form $V_{r,1^{r-2}}\boxtimes V_{1^{n-r+2}}$.

It now remains to analyze the irreducible decomposition of $Res_{S_2\wr S_q}^{S_{2q}}$ of a hook.  The needed calculation, modulo Frobenius reciprocity, is explicitly presented in \cite[Proposition 2.3' (iv)]{KT3} (see also \cite{KT1}) which says that the number of copies of $L_{q}$ appearing in the restriction of a hook $Res_{S_2\wr S_q}^{S_{2q}}(V_{x,1^y})$ is $0$ unless $x=q+1$ and $y=q-1$, in which case it is $1$.  This completes the proof.

While it was convenient that the calculation we needed was available in the literature, there is an argument %\footnote{well this only works with the right parity (q odd?)  That's enough for our purposes, but not ideal.} 
internal to this article which may also be used to prove this.  One first uses the Pieri rule to see that a summand $V_\alpha\boxtimes V_\beta$ appearing in Equation $\ref{res}$ has an $(S_2)^{q}\subset S_2\wr S_q\subset S_{2q}$ invariant subspace if and only if $q=r-1$, $\alpha=(r,2q-r)$ and $\beta=(1^{n+r-2q})$.  Any summand of type $L_q\boxtimes -$ would restrict to an $(S_2)^{q}$ invariant subspace, which establishes the second statement.  On the other hand when $q=r-1$ there is a unique $(S_2)^{r-1}$ invariant subspace in $V_{r,1^{r-2}}$.  Since it is unique, it must contain the image of the gluing operation which grafts $H_0(\Gamma_{0,3})$ onto each input of $H_{r-2}(\Gamma_{1,r-1})\cong V_{1^{r-1}}$, landing in $H_{r-2}(\Gamma_{1,2r-2})\cong V_{r,1^{r-2}}$.  This image is non-zero by Lemma $\ref{gamma}$.  Since we're gluing on to the alternating representation $H_{r-2}(\Gamma_{1,r-1})$, it must be the case that this $(S_2)^{r-1}$ invariant subspace lifts to a $S_2 \wr S_{r-1}$ representation which is alternating with respect to the $S_{r-1}$ factor, i.e.\ to a copy of $L_{r-1}$, which establishes the first statement.  
\end{proof}

\begin{corollary}\label{loopcor}  	
Let $\gamma\in \mathsf{Gr}(g,n)$ with a vertex $v$ of genus $g(v)=1$ and valence $|a^{-1}(v)|=m$.  Consider a homogeneous element 
$$[x] \in \wscmo(\gamma)\tensor_{Aut(\gamma)} \mathfrak{K}^\ast(\gamma)\subset \mathsf{B}(\wscmo)(g,n)$$ whose vertex $v$ carries a label in $\wscmo(v)\cong H_\ast(\Gamma_{1,m})$ of degree $i$.  If $v$ is adjacent to $m-i$ or more loops then $[x]=0$.
%In other words, the maximum number of loops adjacent to a vertex with a degree $i$ label is $m-i-1$. 
\end{corollary}

\begin{proof}  
	Let $q$ be the number of loops adjacent to $v$.  Then $Aut(\gamma)$ contains a subgroup isomorphic to $S_2\wr S_q$ generated by transposing the pair of flags in a loop and permuting the set of loops.  This subgroup acts on $H_i(\Gamma_{1,a^{-1}(v)})\tensor \mathfrak{K}^\ast(\gamma)$ and since
	$H_i(\Gamma_{1,a^{-1}(v)})\cong V_{m-i,1^{i}}$,
the invariants of this action correspond to the copies of $L_q$ appearing in $Res_{S_{2}\wr S_q}^{S_{2q}}(V_{m-i,1^i})$. 	From the proof of Lemma $\ref{replemma}$ we see that the number of such copies is $0$ when $q\geq m-i$.  Any such $Aut(\gamma)$-invariant element $[x]$ would require an $S_2\wr S_q$ invariant element of $H_i(\Gamma_{1,a^{-1}(v)})\tensor \mathfrak{K}^\ast(\gamma)$ labeling $v$.  Since there are no such elements when $q\geq m-i$, the $Aut(\gamma)$-coinvariants vanish.
\end{proof}

\begin{definition}  For a non-negative integer $j$ we define $\theta_{2j+1} \in \mathsf{Gr}(2j+1,0)$ as follows. It has two vertices; one of genus 0, call it $v_0$, and one of genus 1, call it $v_1$.  It has $2j+1$ edges, $3$ of which connect the two vertices and the remaining $2j-2$ of which are loops connected to the vertex of genus $1$.  See Figure $\ref{beta}$.  
\end{definition}

Recall that $\mathsf{B}(\wscmo)(g,n)^{r,s}$ denotes the bigraded component of the chain complex $\mathsf{B}(\wscmo)(g,n)$ having $r$ edges and internal degree $s$.
\begin{lemma}\label{1dim}  The subspace
	\begin{equation*}
	(\mathfrak{K}^\ast(\theta_{2j+1})\tensor_{Aut(\theta_{2j+1})} \wscmo(\theta_{2j+1}))^{2j+1,2j} \subset \mathsf{B}(\wscmo)(2j+1,0)^{2j+1,2j}
	\end{equation*}
	is 1-dimensional.
\end{lemma}
\begin{proof}  Observe that $Aut(\theta_{2j+1})\cong S_3 \times (S_2\wr S_{2j-2})$; the $S_3$ permutes the non-loop edges, the factors of $S_2$ transpose the flags in a loop and the $S_{2j-2}$ permutes the loop edges.   As an $Aut(\theta_{2j+1})$-module, the one dimensional vector space $\mathfrak{K}^\ast(\theta_{2j+1})$ has representation type isomorphic to $V_{1,1,1}\boxtimes L_{2j-2}$.    Thus, the $Aut(\theta_{2j+1})$ fixed points of $\mathfrak{K}^\ast(\theta_{2j+1})\tensor \wscmo(\theta_{2j+1})$ are given by the number of copies of $V_{1,1,1}\boxtimes L_{2j-2}$ in the irreducible decomposition of $V_{2j-1,1^{2j}}\cong \wscmo(1,a^{-1}(v_1))$.  Applying Lemma $\ref{replemma}$ with $r= 2j-1$ and $n=2j$, we see there is exactly one such copy. 
\end{proof}

\subsection{Definition of $\beta$.}
We now construct a canonical basis vector spanning $(\mathfrak{K}^\ast(\theta_{2j+1})\tensor_{Aut(\theta_{2j+1})} \wscmo(\theta_{2j+1}))^{2j+1,2j}$.  Define the set $X:=a^{-1}(v_1)$.  This set is partitioned by the edges of $\theta_{2j+1}$ into three blocks of size 1 and $2j-2$ blocks of size 2.  Choose an auxiliary order on the set of edges of $\theta$ such that the three non-loop edges are in the first three positions, and choose an order on the flags within each block.  This fixes a total order on the set $X$ and hence an isomorphism:
\begin{equation}\label{iso}
H_{2j}(\Gamma_{1,4j-1})\cong H_{2j}(\Gamma_{1,X}).  
\end{equation}

Define $x_{2j+1} \in H_{2j}(\Gamma_{1,X})$ to be the composition $$(...(\alpha_{2j+1}\circ_{2j+1} m_2) \circ_{2j} m_2 ) ... ) \circ_{4} m_2 \in H_{2j}(\Gamma_{1,4j-1})$$ composed with this isomorphism. Here $m_2$ is the generator of $H_0(\Gamma_{0,3}) = \mathsf{Com}(2)$, and $x_{2j+1}\neq 0$ by Lemma $\ref{gamma}$.  Observe that permuting the two flags on $m_2$ acts by $+$, where-as permuting blocks of the permutation on $X$ by $\sigma$ is the same as composing with $\sigma \alpha_{2j+1} = sgn(\sigma)\alpha_{2j+1}$ (by equivariance of the operadic compositions).  Therefore the element $x_{2j+1}$ spans the unique copy of $V_{1,1,1}\boxtimes L_{2j-2}$ in $Res^{S_{4j-1}}_{S_3\times (S_2\wr S_{2j-2})} H_{2j}(\Gamma_{1,X})$, where the $S_{4j-1}$ action is inherited from the isomorphism in Equation $\ref{iso}$.  The class $x_{2j+1} \in H_{2j}(\Gamma_{1,X})$ depends on the choice of isomorphism in Equation $\ref{iso}$, but only up to sign.

Define 
\begin{equation}\label{betadef}
\beta_{2j+1}\in(\mathfrak{K}^\ast(\theta_{2j+1})\tensor_{Aut(\theta_{2j+1})} \wscmo(\theta_{2j+1}))^{2j+1,2j} 
\end{equation}
to be the element formed by labeling $v_1$ with $x_{2j+1}$ and $v_0$ by $m_2$ and using the mod 2 edge order induced by the choice above. Observe that $\beta_{2j+1}$ is independent of the choices made.  If we had picked a different edge order the result would differ by two factors of the sign of the corresponding permutation; if we had picked a different loop orientation the result would differ by a transposition of the commutative product.

Here are some features of $\beta_{2j+1}$ and its underlying graph $\theta_{2j+1}$ for particular values of $j$:
\begin{center}
	\begin{tabular}{c|c|c|c|c|c}
		$j$ & total genus & int.\ deg.\  & loops  & valence of $v_1$ & Rep.\ type at $v_1$ \\ \hline
		$1$ & 3 & 2 & 0 & 3 & $V_{1,1,1}$ \\ 
		2 & 5 & 4 & 2  & 7 & $V_{3,1^4}$ \\ 
		3 & 7 & 6 & 4 &  11 & $V_{5,1^6}$ \\ 
%		4 & 9 & 8 & 6  & 15 & $V_{7,1^8}$  \\ 
		$j$ & $2j+1$ & $2j$ & $2j-2$  & $4j-1$ & $V_{2j-1,1^{2j}}$  \\ 
	\end{tabular}
\end{center}
When $j$ is fixed we may abbreviate the notation $\theta:=\theta_{2j+1}; \beta:=\beta_{2j+1}; \omega:=\omega_{2j+1}$ and so on.

%\begin{equation*}
%X \cong \{a_1,a_2,a_3,l^1_1,l^1_2,l^2_1,l^2_2\cdc l^{2j-2}_1,l^{2j-2}_2 \}
%\end{equation*}

	\begin{figure}
	\includegraphics[scale=.4]{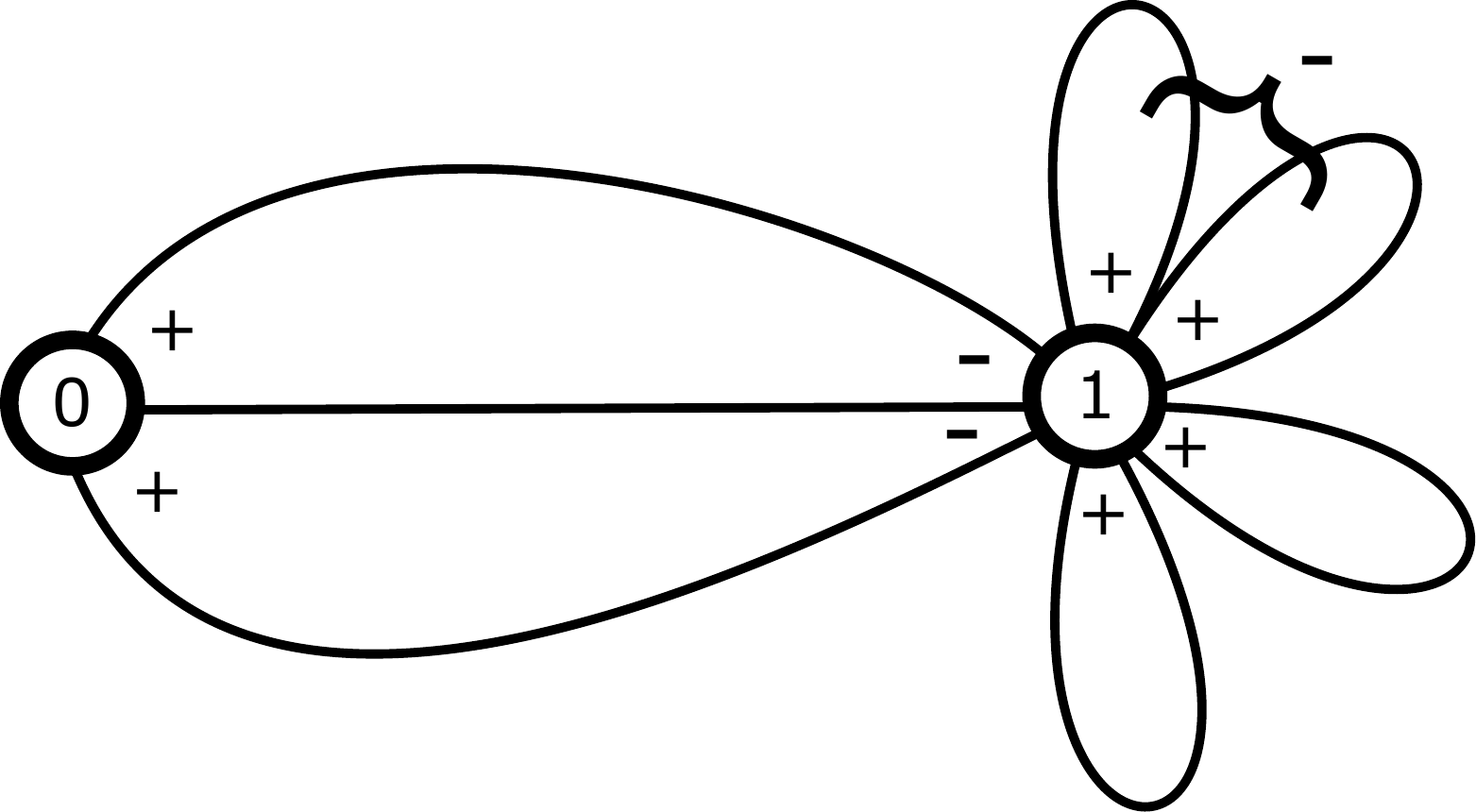} 
	\caption{The case $j=3$.  The graph $\theta_7$ (pictured) underlying the element $\beta_7 \in \mathsf{B}(\wscmo)(7,0)$.  The signs indicate the effect of transposing at the vertices.  The bracket labeled by a minus sign indicates that the block permutation exchanging pairs of flags on two loops gives a minus sign.  %Since the Feynman transform mods out by automorphisms, and since the output of this Feynman transform is $\mathfrak{K}$-twisted, these signs are necessary for this element to be non-trivial.  
		The label of the genus $1$ vertex is (up to scalar multiple) the unique class for which the $Aut(\theta_7)$ acts by such signed multiples.}
	\label{beta}
\end{figure}

\subsection{Analysis of $\partial^{-1}(\beta)$.}

We define a loop in a stable graph $\gamma$ to be simple if the vertex $v$ to which the loop is adjacent satisfies $|a^{-1}(v)|=3$ and $g(v)=0$.
\begin{definition}\label{Bhat}
Define $\widehat{\mathsf{Gr}}(g,n)\subset \mathsf{Gr}(g,n)$ to be the subset of graphs which do not contain simple loops. %$\gamma$ such that:
%\begin{enumerate}
%		\item $\gamma$ has no vertex of type $(1,1)$.
%		\item If $v$ is a vertex of $\gamma$ of type $(0,3)$, then $v$ is not adjacent to a loop in $\gamma$.
%	\end{enumerate}
For a weak modular operad $A$ we then define:
\begin{equation*}
\widehat{\mathsf{B}}(A)(g,n) = \ds\bigoplus_{\gamma \in \widehat{\mathsf{Gr}}(g,n) } A(\gamma)\tensor_{Aut(\gamma)} \mathfrak{K}^\ast(\gamma) \subset \mathsf{B}(A)(g,n).
\end{equation*}	
\end{definition}

\begin{lemma}\label{subcomplex} The submodule $\widehat{\mathsf{B}}(A) \subset \mathsf{B}(A)$ is closed under the co-Feynman transform differential $\partial$.
\end{lemma}
\begin{proof}  Given a modular graph $\gamma$ with no simple loops, terms in the weak co-Feynman transform differential are indexed by contractions of subgraphs of $\gamma$.  If such a differential term has a simple loop it must have been created by contracting a subgraph of type $(0,3)$.  However $\mathsf{Gr}(0,3)$ consists only of the $(0,3)$-corolla, so no such contraction is possible.
\end{proof}

We remark that since the weak modular operad $\wscmo$ has internal differential $0$, the summand of $\partial$ which contracts just 1 edge, call it $\partial_1$, is itself a differential (Remark $\ref{d1rmk}$).  %This differential $\partial_1$ depends only on the underlying modular operad structure of Lie graph homology and doesn't see any higher operations.  
In this case Lemma $\ref{subcomplex}$ also shows $(\widehat{\mathsf{B}}(\wscmo)(g,n),\partial_1)$ is a subcomplex of $(\mathsf{B}(\wscmo)(g,n),\partial_1)$. % Observe that $\beta_{2j+1}$ is contained in this subcomplex.

\begin{proposition}\label{propzero}  The composition of $\partial_1$ with projection to the $\theta_{2j+1}$ summand of $\widehat{\mathsf{B}}(\wscmo)(2j+1,0)^{2j+1,2j}$ is zero.  In particular, $\beta_{2j+1}$ does not appear with non-zero coefficient in any $\partial_1$ boundary. 
%	Every $\partial_1$ boundary in $\widehat{\mathsf{B}}(\wscmo)(2j+1,0)$ 	$\beta_{2j+1} \in \widehat{\mathsf{B}}(\wscmo)(2j+1,0) $ is not a $\partial_1$ boundary.
\end{proposition}

\begin{proof} As above let $X:=a^{-1}(v_1)$.  The set $X$ is partitioned by the edges of $\theta_{2j+1}$ into 3 blocks of size one and $2j-2$ blocks of size 2.  %Label the elements of $X$ by the symbols:
%\begin{equation*} X \cong \{a,b,c,l^1_1,l^1_2,l^2_1,l^2_2,...\} \end{equation*}
Label the blocks of size $1$ by $a_1,a_2,a_3$ and label the elements of each block of size two by $l^i_1,l^i_2$ label, where $i$ indexes the $2j-2$ loops of $\theta_{2j-2}$.  %This isomorphism gives us a total order on $X$, and in particular fixes an isomorphism $S_3\times (S_2\wr S_{2j-2})\cong Aut(\theta_{2j+1})\subset S_X$ which we leave implicit below.  
As above we write $x_{2j+1}\in H_{2j}(\Gamma_{1,X})$ for a basis vector spanning the unique invariant subspace of $Res^{S_X}_{Aut(\theta_{2j+1})}$ isomorphic to $V_{1,1,1}\boxtimes L_{2j-2}$.  The $S_{\{a_1,a_2,a_3\}}$ action on $x_{2j+1}$ is alternating while each $S_{\{l^i_1,l^i_2\}}$ action on $x_{2j+1}$ is the identity.

Write $\partial_{1}^{\theta}$ for the composition of $\partial_1$ with projection to the $\theta_{2j+1}$ summand of $\widehat{\mathsf{B}}(\wscmo)(2j+1,0)$.  %, so:
%\begin{equation*} \widehat{\mathsf{B}}(\wscmo)(2j+1,0)^{2j+1,2j+1} \stackrel{\partial_{1}^{\theta}}\longrightarrow	(\mathfrak{K}^\ast(\theta_{2j+1})\tensor_{Aut(\theta_{2j+1})} \wscmo(\theta_{2j+1}))^{2j+1,2j}   \end{equation*}
Define $\Theta$ to be the set of graphs $\gamma$ for which the following composite is non-zero:
\begin{equation*}\xymatrix{
(\mathfrak{K}^\ast(\gamma)\tensor_{Aut(\gamma)} \wscmo(\gamma))^{2j+2,2j} \ar@{-->}[dr]  \ar@{^(->}[r] & \widehat{\mathsf{B}}(\wscmo)(2j+1,0)^{2j+2,2j} \ar[d]^{\partial_{1}^{\theta}} \\  & (\mathfrak{K}^\ast(\theta_{2j+1})\tensor_{Aut(\theta_{2j+1})} \wscmo(\theta_{2j+1}))^{2j+1,2j}    }
\end{equation*}
To prove the claim it is sufficient to show that the set $\Theta$ is empty.  By way of contradiction, suppose $\gamma \in \Theta$.  Then $\gamma$ has an edge $e = \{r,s\}$ such that $\gamma/e \cong \theta$.

Note that the vertex of $\gamma/e$ corresponding to $e$ must be sent to the vertex $v_1$ of $\theta_{2j+1}$ since the vertex $v_0$ is of type (0,3) and hence indecomposible.  Note also that the edge $e$ of $\gamma$ can not be a loop for degree reasons -- such a $\gamma$ would have only genus 0 vertices, and so internal degree $0\neq 2j+1$.

Therefore the edge $e$ of $\gamma$ is adjacent to two vertices whose genera add to 1.  Let $Y\cup \{r\}$ be the flags of $\gamma$ adjacent to the genus 1 vertex and $Z\cup \{s\}$ be the flags of $\gamma$ adjacent to the genus 0 vertex.  The isomorphism $\gamma/E \cong \theta$ specifies a partition
	\begin{equation*}
	X =Y\sqcup Z,
	\end{equation*}
	along with a linear map
	\begin{equation*}
	H_{2j}(\Gamma_{1,Y\cup \{r\}})\tensor H_0(\Gamma_{0,Z\cup \{s\}}) \stackrel{\circ_e}\longrightarrow H_{2j}(\Gamma_{1,X}).
	\end{equation*}
	This linear map is $S_{Y}\times S_Z$ equivariant.   In particular, $|X|=|Y|+|Z| = 4j-1$.  We say a loop $l^i$ of $\theta$ is split (by $e$) if both $\{l^i_1,l^i_2\}\cap Y$ and $\{l^i_1,l^i_2\}\cap Z$ are nonempty.  
	
Such a differential term being nonvanishing implies the following:
	\begin{itemize}
		\item $|Y\cup r| \geq 2j+1$ and hence $|Y|\geq 2j$.  Thus $3\leq |Z\cup s| \leq 2j$.  
		\item $|\{a_1,a_2,a_3\}\cap Z| < 2$, since the representation type of $H_0(\Gamma_{0,Z\cup \{s\}})$ is trivial.  So we suppose without loss of generality that $a_2,a_3\in Y$.
		\item At most one loop is split, since otherwise we  would create parallel edges with an alternating action of 
			$\{l_1^i, l_1^h\}$ at one vertex	
		an identity action of
		$\{l_2^i, l_2^h\}$ at the other vertex
		which pass equivariantly to the identity action of both
		$\{l_1^i, l_2^i\}$ and $\{l_1^h, l_2^h\}$
on $X$.%, which is an obvious contradiction.
	\end{itemize}

	Let $\ell$ be the number of loops adjacent to $Y$.  By Corollary $\ref{loopcor}$ we know $\ell \leq |Y|+1 -2j-1=|Y|-2j$, and hence
	\begin{equation*}
	2j \leq |Y|-\ell.
	\end{equation*}  
Consider the possible cases for such a $\gamma \in \Theta$.% supporting a non-zero composition with $\partial_{1}^\theta$.
	
	{\bf Case} $a_1\in Z$ and no loop is split:  Then $|Y|=2\ell+2$ which implies $2j-2 \leq \ell$ and hence  $2j-2=\ell$, since $2j-2$ is the total number of loops.  But then the arity of $Z$ would be 2, contradiction.
	
	{\bf Case} $a_1\in Z$ and one loop is split:  Then $|Y|=2\ell+3$ which implies $2j-3 \leq \ell$ and hence  $2j-3=\ell$, since $2j-2$ is the total number of loops and one was split so can't be adjacent to $Y$. 
	This means that $Y$ carries the maximum number of loops for its given valence, so is alternating off the loops (Lemma $\ref{replemma}$).  Let $l_1,l_2$ be the split loop with $l_1\in Y$ and $l_2\in Z$.  Thus the action of $l_1,a_2$ is alternating on $Y$ and hence on $X$.  But the action of $a_1,a_2$ is alternating whilst $a_1, l_2$ is id on $Z$ and hence $X$, which is a contradiction of the fact that the transpositions $(a_1a_2)$, $(a_1l_2)$ and $(l_1l_2)$ generate the symmetric group of $a_1,a_2,l_1,l_2$.
	
	So we conclude $a_1\in Y$ and proceed to:

	{\bf Case:} Suppose one loop is split.  Then $|Y|= 2\ell+4$ and so $2j-4 \leq \ell < 2j-2$, but only one loop was split so by stability considerations, one must go on $Z$, hence $\ell=2j-4$.  %, which is the maximum number of loops for that arity and hence $Y$ is alternating off the loops.  
	Let $l^i$ be the split loop and $l^h$ be the loop on $Z$.  Equivariance will imply that all permutations of $\{l_1^i,l_2^i,l_1^h,l_2^h\}$ must act by the identity on $X$, (since we can switch $l^i_2,l^h_2$ adjacent to $Z$).  This contradicts the definition of $x_{2j+1}$ which says that $(l^i_1 l^h_1)(l^i_2 l^h_2)$ must act by $-1$.

	We thus conclude no loop is split, hence $|Y|=2\ell+3$ and so $2j-3\leq \ell \leq 2j-2$.
	But if $\ell = 2j-2$, then the vertex adjacent to $Z$ would be unstable.  So the only remaining possibility is that $\ell = 2j-3$, which in turn implies that $Z$ is a vertex of valence 3, genus 0 and adjacent to 1 loop.  But such graphs are excluded from $\widehat{\mathsf{B}}(\wscmo)$ by definition.  We thus conclude $\Theta$ is empty, hence $\beta_{2j+1} \in \widehat{\mathsf{B}}(\wscmo)(2j+1,0)$ is not a boundary. \end{proof}
We remark that the subcomplex $\widehat{\mathsf{B}}(\wscmo)(2j+1,0)\subset \mathsf{B}(\wscmo)(2j+1,0) $ 
does not split, and we do not assert that $\beta_{2j+1}$ is a non-boundary when viewed in $\mathsf{B}(\wscmo)(2j+1,0)$, although the above proof shows that its inverse image is supported on a 1-dimensional subspace of $\mathsf{B}(\wscmo)(2j+1,0)$.

\begin{corollary}\label{deg0}  
	Let $\xi\in\widehat{\mathsf{B}}(\wscmo)(2j+1,0)^{4j+2-s,s}$ be a vector of positive internal degree $s>0$.  Then the projection of $\partial(\xi)$ to the $\theta_{2j+1}$ summand is zero.  %In particular, $\beta_{2j+1}$ does not appear with non-zero coefficient in any $\partial$ boundary of positive internal degree.	
%	$\partial^{-1}(\beta_{2j+1}) \subset \widehat{\mathsf{B}}(\wscmo)^{4j+2} $ is supported on internal degree $0$.
\end{corollary}
\begin{proof}  The Lemma establishes the case $s=2j$. 
	Suppose $0<s<2j$.  Without loss of generality we may assume $\xi$ is a homogeneous element supported on a summand index by a modular graph $\gamma$.   A term in $\partial(\xi)$ is non-zero upon projection to the $\theta_{2j+1}$ summand only if  
it is possible to contract a subgraph $\gamma^\prime$ such that $\gamma/\gamma^\prime \cong \theta$ where $\gamma^\prime$ has $2j+1-s>1$ edges.  As above, such an isomorphism must send the vertex corresponding to $\gamma^\prime$ to $v_1$, so the total genus of $\gamma^\prime$ must be $1$. By definition of $\wscmo$, such an operation is non-trivial only if $\gamma^\prime$ has first betti number $1$.  These two conditions are true simultaneously only if each vertex of $\gamma^\prime$ has genus $0$, which in turn implies that each vertex carries a label in some $H_\ast(\Gamma_{0,m})$, which is concentrated in internal degree $0$.  The only other vertex of $\gamma$ has genus $0$ as well, hence such an element must be supported on internal degree $0$.
\end{proof}

\subsection{Co-operadic non-zero coefficient lemma}

\begin{lemma}\label{notzero}  The differential of the wheel graph 	$\partial(\omega_{2j+1})$ contains $\beta_{2j+1}$ with non-zero coefficient.
\end{lemma}

\begin{proof} By Lemma $\ref{1dim}$, it suffices to show that the projection of $\partial(\omega_{2j+1})$ to the $\theta_{2j+1}$ summand of $\mathsf{B}(\wscmo)(2j+1,0)^{2j+1,2j}$ is non-zero. After Diagram $\ref{coftd}$, it is sufficient to show that composition in the following diagram is not zero: 
\begin{equation}\label{nz1}
\xymatrix{ \omega_{2j+1} \in \mathsf{B}(\wscmo)(2j+1,0)^{4j+2,0} \ar[d]^{\nu_\theta} \ar[r]^{ \ \ \ \ \partial} & \mathsf{B}(\wscmo)(2j+1,0)^{2j+1,2j} \ar@{->>}[d]^{\pi_{\theta}}  \\ 
	\mathsf{B}(\wscmo)(\theta_{2j+1})^{2j+1} \tensor_{Aut(\theta)} \mathfrak{K}^\ast(\theta_{2j+1}) \ar[r]^{\ \ \ \bar{\mu}}	&  \wscmo(\theta_{2j+1})^{2j} \tensor_{Aut(\theta)}\mathfrak{K}^\ast(\theta_{2j+1})
}
\end{equation}	

By definition, $\nu_\theta(\omega_{2j+1})$ is determined by summing over ways to nest the graph $\omega_{2j+1}$ such that collapsing nests gives the graph $\theta_{2j+1}$.  Such a nesting specifies two induced graphs of $\omega_{2j+1}$, $\hat{N}_0$ and $\hat{N}_1$ which collapse to the two vertices $v_0$ and $v_1$ of $\theta_{2j+1}$.  Since the vertex $v_0$ is of type $(0,3)$, the induced stable graph $\hat{N}_0$ must be a corolla and $N_0$ must consist of a lone vertex.    Given such an $N_0$, there is a unique such $N_1$ for which the composite with $\mu_X$ is non-zero. 
% That is because the map $\mu_X$ is given by applying Massey product corresponding to $N_1$:
%\begin{equation*}
%\mathsf{B}(\wscmo)(1,X)^{2j+1}\supset \wscmo(\hat{N}_1)\stackrel{\mu_{N_1}}
%\longrightarrow \wscmo(1,X)
%\end{equation*}
%and then applying $- \tensor_{Aut(\gamma)}\mathfrak{K}^\ast(\gamma)$, and by definition of the weak $\mathfrak{K}$-modular operad structure on $\wscmo$, such a contraction is non-zero only if $N_1$ is a $2j+1$-gon.  
It is given by the unique $2j+1$-gon $N_1$ which is a subgraph of $\omega_{2j+1}$ and which misses a distinguished vertex $N_0$ (right hand side of Figure $\ref{fig:w5}$). 

When $j=1$, a direct computation shows that the four choices for a lone vertex $N_0$ are sent by $\nu_{\theta_3}$ to the same element, hence $\nu_{\theta_3}\neq 0$.  Indeed, with the convention that the edges of the nested triangle appear last in the decomposition, one must apply the permutations $(14)$,$(25)$,$(36)$ and $(14)(25)(36)$ to the conventional edge ordering of $\omega_3$ (Figure $\ref{fig:w5}$) to decompose, and these permutations are all odd.  Then since $\alpha_3=x_3$, $\mu_X$ is simply the Massey product which contracts the triangle, which is not zero.  Whence the case $j=1$.

So we now assume $j>1$. In this case there are $2j+1$ choices for such an $N_0$, corresponding to the outer vertices of $\omega$. 
Therefore, $\nu_{\theta}(\omega)$ is a sum of $2j+1$ terms, corresponding to the choice of an outer vertex and the complimentary $2j+1$ gon.  Since these terms are related by an automorphism of $\omega$, it is enough to show that any one of them is non-zero when composed with $\mu_X$.

So let us fix such an $N_0$ and $N_1$.  The term in the sum $\nu_{\theta}(\omega)$ corresponding to this choice of nesting is given by choosing an isomorphism $\omega/N_1 \cong \theta$, which in turn specifies a labeling of the flags of $\hat{N}_1$ by the set $X:=a^{-1}(v_1)$.  We import the notation $$X=a^{-1}(v_1) = \{a_1,a_2,a_3,\ell^i_1,\ell^i_2 \ | \ 1\leq i\leq 2j-2\}$$ from the proof of Proposition $\ref{propzero}$.   
Since $\omega$ has a unique non-trivalent vertex, so does $\hat{N_1}$. The conditions on this $X$-labeling of the flags of $\hat{N_1}$ coming from the isomorphism $\omega/N_1\cong \theta$ are that the non-trivalent vertex of $\hat{N}_1$ must be adjacent to flags labeled by exactly one of the $a_i$, and one flag from each loop.  The two vertices adjacent to the non trivalent vertex must have the other $a$ labels (Figure $\ref{fig:p}$).  Let $P\subset \mathsf{B}(\wscmo)(1,X)$ be the span of such $X$-labeled, $2j+1$-gons. % They are odd, they have one non-trivalent vertex, they are $X$-labeled subject to the condition that gluing along $\gamma$ gives $\omega$ (as above).
In particular $dim(P)=3!(2j-2)!2^{2j-2}/2$.  We conclude that the bottom row of the diagram is supported on the restriction
\begin{equation*}
\nu_\theta(\omega_{2j+1})_{|_P} \in (\mathsf{Com}(2)\tensor P)\tensor_{Aut(\theta)}\mathfrak{K}^{\ast}(\theta_{2j+1}) \subset \mathsf{B}(\wscmo)(\theta_{2j+1}) \tensor_{Aut(\theta)} \mathfrak{K}^\ast(\theta_{2j+1})
\end{equation*}
where $\mathsf{Com}(2)$ labels $v_0$ and $P$ labels $v_1$.  

The bottom row in Diagram $\ref{nz1}$ is given by contracting the $X$-labeled graph at the vertex $v_1$ of $\theta_{2j+1}$, via the associated Massey product.  It thus remains to show that contraction of such $X$-labeled polygons
\begin{equation*}
(\mathsf{Com}(2)\tensor P)\tensor\mathfrak{K}^\ast(\theta_{2j+1}) \stackrel{\mu_X}\to \wscmo(\theta_{2j+1})^{2j}\tensor \mathfrak{K}^\ast(\theta_{2j+1})
\end{equation*}
is non-zero upon passage to $Aut(\theta)$-coinvariants.  Since this map is $Aut(\theta)$ invariant, it is enough to know that the class $x_{2j+1}\in \wscmo(1,X)$ is in the image of the contraction when restricted to  $P\subset \mathsf{B}(\wscmo)(1,X) \to \wscmo(1,X)$.  The remainder of the proof is dedicated to this calculation.

%The map $\mu_X$ in the bottom row of Diagram $\ref{nz1}$ is induced by contracting the $X$-labeled graph labeling $v_1$.    $P\subset \mathsf{B}(\wscmo)(1,X) \to \wscmo(1,X)$

%Since the source and target of this map are both 1-dimensional, this will be true if and only if the contraction map $P\subset \mathsf{B}(\wscmo)(1,X) \to \wscmo(1,X)$
%surjects onto the span of $x_{2j+1} \in H_{2j}(\Gamma_{1,X})$.  % The passage to coinvariants will still be surjective.

\begin{figure}
	\centering
	\includegraphics[scale=0.45]{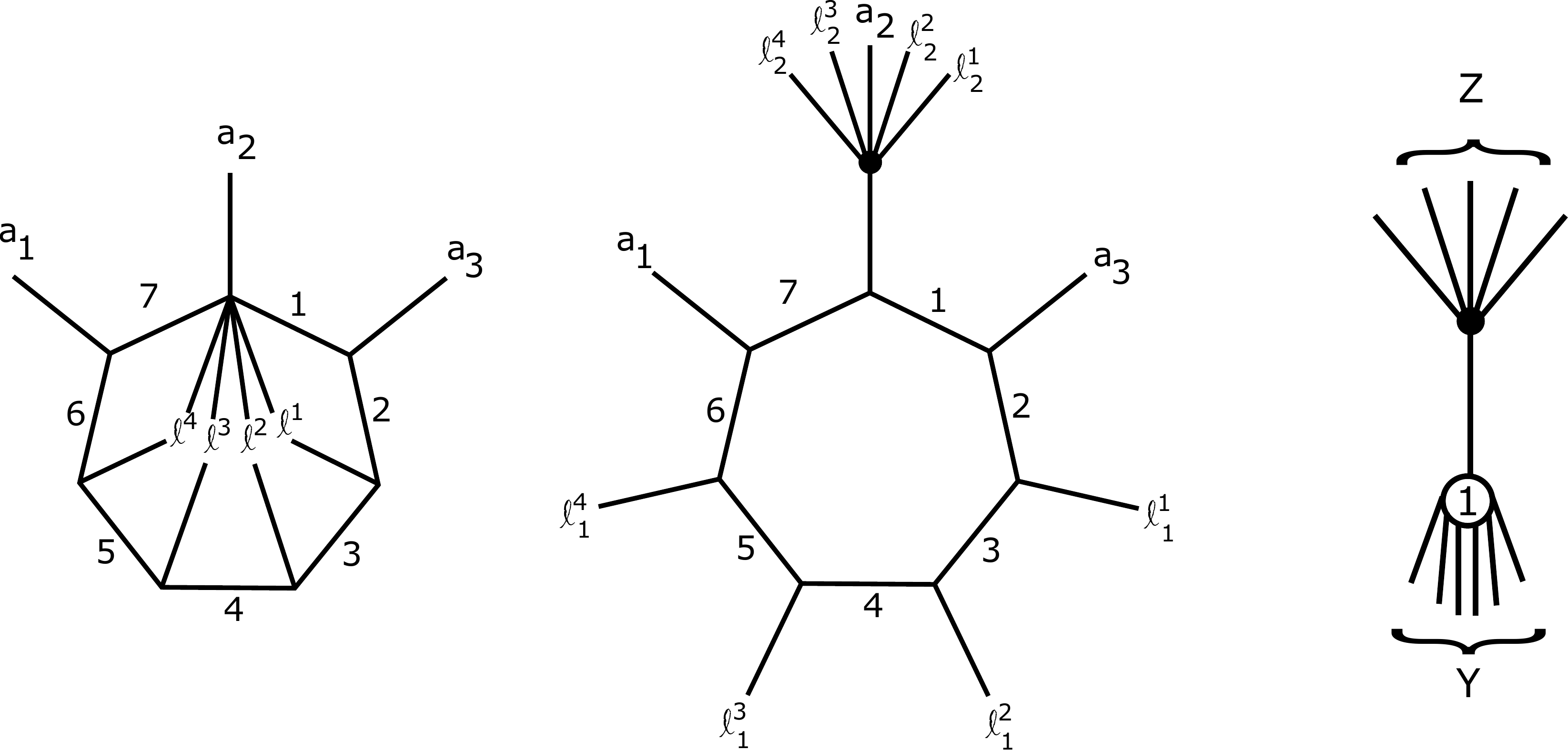}
	\caption{At left, the vector $\rho$ in the case $j=3$ is a graph with $11$ legs whose underlying leg free graph is a $7$-gon.  The image of $\rho$ under the heptagonal Massey product is determined by first contracting the heptagon in the center to the VCD class $\alpha_7\in H_{6}(\Gamma_{1,Y})$ and then contracting the lone edge in the graph on the right. }
	\label{fig:p}
\end{figure}

Recall that the class $x_{2j+1}$ is defined by choosing a total order on the set $X$, which fixes an isomorphism $H_{2j}(\Gamma_{1,4j-1})\cong H_{2j}(\Gamma_{1,X})$ (see Equation $\ref{iso}$).  By abuse of notation we write $x_{2j+1}\in H_{2j}(\Gamma_{1,4j-1})$ for the image of $x_{2j+1}$ under this isomorphism.  Note that since $x_{2j+1}$ depends on this isomorphism only up to sign, the span is independent of choice.  Thus, using the bijection $X\cong \{1\cdc 4j-1\}$,
 we may view $P$ as carrying numerical labels, and it suffices to show the contraction
\begin{equation*}
P \subset \mathsf{B}(\wscmo)(1,4j-1) \to H_{2j}(\Gamma_{1,4j-1})
\end{equation*}
surjects onto the span of $x_{2j+1}$. To be pedantic, the bijection $X\cong \{1\cdc 4j-1\}$ sends $a_i \mapsto i$ for $1\leq i \leq 3$ and $\ell_r^i\mapsto 2(i+1)+(r-1)$.  As such we call a pair of numbers $2(i+1), 2(i+1)+1$ with $1\leq i \leq 2j-2$, a loop, called loop $i$ or the $i^{th}$ loop.  We say a representative of a loop is one of its entries.

%As above, $P$ be the space of odd $2j+1$-gons having $4j-1$ legs which glue along $\theta$ to produce $\omega$.  Each such $2j+1$-gon has a distinguished vertex adjacent to $2j-1$ legs, labeled with a representative of each loop and one of $\{1,2,3\}$.  The two adjacent vertices are trivalent and adjacent to legs labeled with the remaining elements of $\{1,2,3\}$.  The remaining vertices are $2j-2$ vertices are also trivalent and adjacent to legs labeled by the remaining loop representatives.

The contraction map $P\to H_{2j}(\Gamma_{1,2j-1})$ is given by the Massey product which contracts each such $2j+1$-gon.  The image of this contraction is determined by expanding the distinguished (non-trivalent) vertex and its legs to a new edge, contracting the $2j+1$-gon, and then contracting the expanded edge (Figure $\ref{fig:p}$).  In other words, this map factors as
\begin{equation}\label{comp2}
P \to Ind^{S_{4j-1}}_{S_{2j-1}\times S_{2j}}(H_{0}(\Gamma_{0,2j})\tensor H_{2j}(\Gamma_{1,2j+1})) \stackrel{\circ_e}\longrightarrow H_{2j}(\Gamma_{1,4j-1})
\end{equation}
Notice the map $\circ_e$ is surjective, since it is non-zero and the target is irreducible.

We now analyze $\circ_e^{-1}(x_{2j+1})$. Recall $x_{2j+1}$ spans a representation of $Aut(\theta)$ for which $S_{\{1,2,3\}}$ acts by the alternating representation, transpositions of representatives of a loop acts by the identity and where each $(45)(2i,2i+1)$ with $1<i-1\leq 2j-2$ acts by $-1$.

The space $Ind^{S_{4j-1}}_{S_{2j-1}\times S_{2j}}(H_0(\Gamma_{0,2j})\tensor H_{2j}(\Gamma_{1,{2j+1}}))$ has a basis given by $  i_1\wedge i_2 \wedge \dots \wedge i_{2j}$ with
where $1\leq i_1< i_{2}<... <i_{2j}\leq 4j-1$.  The $S_{4j-1}$ action is by permutation, with permutation of wedge products acting by the sign of the permutation.  Write a preimage of $x_{2j+1}$ in this basis:
	\begin{equation}\label{pullback}
\ds\sum_{1 \leq i_1<...<i_{2j}\leq 4j-1} c(i_1, i_2 \cdc i_{2j}) i_1\wedge i_2 \wedge \dots \wedge i_{2j} \in \circ_e^{-1}(x).
	\end{equation}
Since $S_{\{1,2,3\}}$ acts by the sign representation on $x_{2j+1}$, and $\circ_e$ is equivariant, $c(i_1\cdc i_{2j})=0$ unless $\{ i_1\cdc i_{2j}\} \cap \{1,2,3\} \geq 2$. % (else could permute the missing pair and get $c=-c$).  
Likewise, $c(i_1\cdc i_{2j})=0$ if it lists both representatives of any loop since transposition of loop representatives acts by the identity.

To convey additional conditions that the coefficients in Equation $\ref{pullback}$ must satisfy we fix some new notation.  For a subset $I\subset \{1\cdc 2j-2\}$ we define $\op{S}_I$ to be the set of lists of representatives of those loops not appearing in $I$, such that no list has two representatives of the same loop.  In particular, each list has $2j-2-|I|$ entries and there are $2^{(2j-2-|I|)}$ such lists.  
Let $S_I$ be the formal sum of wedge products appearing in $\op{S}_I$.  For brevity we write $\op{S}:=\op{S}_\emptyset$, $S:=S_\emptyset$, $\op{S}_i:=\op{S}_{\{i\}}$ and $S_i:=S_{\{i\}}$. We will denote lists of loop representatives by $\vec{s}\in \op{S}_I$.  By abuse of notation we also write $\vec{s}$ for the associated wedge product.  We write $\vec{s}^\perp$ for the list of complimentary representatives. 
  For example if $j=2$ then $\op{S} = \{(4,6),(5,6),(4,7),(5,7)\}$, $\op{S}_2 = \{(4),(5)\}$, $S = 4\wedge 6 + 5\wedge 6 + 4\wedge 7 + 5\wedge7$ and $S_2 = 4\wedge 5$.  If $\vec{s}=(4,6)$ then $\vec{s}^\perp = (5,7)$, and by abuse of notation me may also write $\vec{s}=4\wedge 6$.

With this notation we describe conditions on the coefficients in Equation $\ref{pullback}$.  First, %if $i_1,i_2,i_3=1,2,3$ then:
			\begin{equation*}
c(1,2,3, \vec{a}) = (-1)^{i-1}c(1,2,3, \vec{b}) \text{ if } \vec{a}\in \op{S}_1 \text{ and } \vec{b} \in \op{S}_i.
	\end{equation*}	
This condition is forced by the equivariance of $\circ_e$. Transposing representatives of a loop acts by the identity, so the coefficient is independent of choice of list in $\op{S}_i$ (resp.\ $\op{S}_1$).  On the other hand, a list in $\op{S}_1$ is missing one loop, namely $\{4,5\}$, and the permutation $(45)(2(i+1),2(i+1)+1)$ acts by $-1$ for any $i>1$. 
 The result may be compared with a list in $\op{S}_i$ by applying an $i-1$ cycle, acting by $(-1)^i$, hence the claim.

Second, consider coefficients whose index has exactly two of 1,2,3.  The condition that both representatives of a loop can't appear in a wedge, means that wedges must have one representative from each loop.  This gives conditions:
	\begin{equation*}
	c(1,2,\vec{a}) = -  c(1,3,\vec{a}) =  c(2,3,\vec{a})  \text{ \ and \ } c(1,2,\vec{a}) = c(1,2,\vec{b}) \text{ for each } \vec{a},\vec{b} \in \op{S}.
	\end{equation*}

 Define $w,v \in Ind^{S_{4j-1}}_{S_{2j-1}\times S_{2j}}(H_0(\Gamma_{0,2j})\tensor H_{2j}(\Gamma_{1,{2j+1}}))$ by	
	\begin{equation}\label{wandv}
	w = \sum_{1\leq i \leq 2j-2} (-1)^{i} 1\wedge 2 \wedge 3 \wedge S_{i} \text{ \ \ \ and \ \ \ } 	v =  1\wedge 2 \wedge S - 1\wedge 3 \wedge S + 2\wedge 3 \wedge S.
	\end{equation}
	The above conditions on coefficients show that a vector in
	$\circ_e^{-1}(x_{2j+1})$ must be a linear combination of $w$ and $v$. 
To find this linear combination, we invoke Lemma $\ref{relations}$ which, in the present notation states 
	\begin{equation}\label{relations2}
	(id+(1,2j)+(2,2j)+\dots+(2j-1,2j)) 2j\wedge 2j+1 \wedge \dots \wedge 4j-1
	\end{equation}
	is in the kernel of $\circ_e$.  Consequently any permutation of this relation is in $ ker(\circ_e)$ as well. % Informally this relation says that a wedge product is the negative of the sum of ways to swap in the elements not appearing in the wedge.  
	Using this relation, we now calculate the linear combination of $v$ and $w$ which is in $ ker(\circ_e)$.   Writing $\sim$ for the induced equivalence relation, the calculation follows from the following two claims:

{\bf Claim 1:}  $
1\wedge 2 \wedge S \sim -1\wedge 3 \wedge S \sim 2\wedge 3 \wedge S \sim  -1\wedge 2 \wedge 3 \wedge S_{1} $

{\bf Proof:} Let's prove that $2\wedge 3 \wedge S \sim  -1\wedge 2 \wedge 3 \wedge S_{1} $, with the other two following similarly. $2\wedge 3 \wedge S $ is a sum of $(2j-2)^2$ wedge products of length $2j$.  Each of the terms in the sum have $4$ or $5$ appearing once in the third position, and $1$ does not appear at all. Hence terms appearing in this sum can be paired to write:
$
 2\wedge 3\wedge S = \sum_{\vec{b}\in S_1} 2\wedge 3 \wedge (4\oplus 5)\wedge \vec{b}
$.  
To each such $\vec{b}$ we apply a permutation of the relation in Equation $\ref{relations2}$ to $2\wedge 3 \wedge 1 \wedge \vec{b}$ to find:
\begin{equation*}
2\wedge 3 \wedge 4 \oplus 5 \wedge \vec{b} =
((14)\oplus (15)) 2\wedge 3 \wedge 1 \wedge \vec{b} \sim -2\wedge 3 \wedge 1 \wedge \vec{b} - 2\wedge 3 \wedge (\oplus_{b\in \vec{b}^\perp} b) \wedge \vec{b}
\end{equation*}
The sum over all $\vec{b}$ of $- 2\wedge 3 \wedge (\sum_{b\in \vec{b}^\perp} b) \wedge \vec{b}$ vanishes since each term has two representatives of one loop, and hence pairs with the term in which the representatives appear in the transposed order.  Combining the previous two equations with this observation proves the claim.

{\bf Claim 2:} $(-1)^{i-1}1\wedge 2 \wedge 3 \wedge S_i \sim 1\wedge 2 \wedge 3 \wedge S_1$.

{\bf Proof:}  The claim is vacuous for $i=1$, so fix $1 < i \leq 2j-2$.  Then each term in $1\wedge 2 \wedge 3 \wedge S_i$ has 4 or 5 appearing in it. Apply the relations to each term with a 4 appearing:
%\begin{equation*} 1\wedge 2 \wedge 3 \wedge 4 \wedge \vec{s}\sim - \sum_{a\in \vec{s}}(4,a) (1\wedge 2 \wedge 3 \wedge a \wedge \vec{s}) \end{equation*}
\begin{equation*}
1\wedge 2 \wedge 3 \wedge 4 \wedge \vec{b}\sim - 1\wedge 2 \wedge 3 \wedge (5\oplus 2(i+1) \oplus (2(i+1)+1) \oplus (\oplus_{b\in \vec{b}^\perp} b)) \wedge \vec{b}
\end{equation*}
for each $\vec{b}\in \op{S}_{\{1,i\}}$. The terms replacing $4$ with $5$ cancel with the terms in which 5 originally appeared.  The terms where both representatives of a loop appear cancel in pairs.  The remaining terms replace 4 with a representative of loop $i$.  To compare these terms with $1\wedge 2\wedge 3 \wedge S_1$, it suffices to permute the order of the representatives of the loops into numerical order.  This is done via a cycle of length $i-1$, 
so produces a factor of $(-1)^{i-2}$.  Combining this with the factor of $-1$ already appearing yields the result.

These two claims together imply $(2j-2)v-3w \in ker(\circ_e)$.  From this calculation we conclude that $\circ_e(v) \neq 0$, since if it were zero, it would imply $\circ_e^{-1}(x_{2j+1})=0$, but iteration of Lemma $\ref{gamma}$ shows this not to be the case.

Finally, it remains to observe that the vector $v$ is in the image of $P$ in Equation $\ref{comp2}$.  Indeed the map $P\to Ind^{S_{4j-1}}_{S_{2j-1}\times S_{2j}}(H_0(\Gamma_{0,2j})\tensor H_{2j}(\Gamma_{1,2j+1})$ sends each $X$-labeled $2j+1$-gon to a wedge containing two of $\{1,2,3\}$ and exactly one representative of each loop.  The vector $v$ is defined as a linear combination of such wedge products, so it is in the image.  In particular, we have shown the contraction map $P \to \wscmo(1,X)$ surjects onto the span of $x_{2j+1}$, from which the statement follows. 	\end{proof}

\subsection{Proof of Main Results}\label{ld}
To conclude, we observe how our main results stated in the introduction follow by linear dualizing the results of this section.
 
{\bf Proof of Theorem $\ref{itsaboundary2}$.}  
As above $\widehat{\mathsf{B}}$ denotes the subcomplex of the co-Feynman transform with no simple loops.  Dualizing the inclusion $\widehat{\mathsf{B}}(\wscmo)\hookrightarrow \mathsf{B}(\wscmo)$, we have a projection $\FT(\wscmo)\twoheadrightarrow \bar{\FT}(\wscmo)$ which quotients by simple loops.  Since $\beta_{2j+1}$ spans the $\theta_{2j+1}$ summand of $\mathsf{B}(\wscmo)(2j+1,0)^{2j+1,2j}$, we may define its characteristic functional $\eta_{2j+1} := \beta_{2j+1}^\ast \in \FT(\wscmo)(2j+1,0)$, extending by $0$ off this summand.  We write $\bar{\eta}_{2j+1} \in \bar{\FT}(\wscmo)(2j+1,0)$ for its image under projection.    Corollary $\ref{deg0}$ shows that no vector in $\partial(\text{ker}( \widehat{\mathsf{B}}(\wscmo) \twoheadrightarrow  \widehat{\mathsf{B}}(\mathsf{Com})))$ has a non-zero coefficient of $\beta_{2j+1}$.  Therefore $d(\bar{\eta}_{2j+1})(\text{ker}( \widehat{\mathsf{B}}(\wscmo) \twoheadrightarrow  \widehat{\mathsf{B}}(\mathsf{Com})))=0$, from which we conclude $d(\bar{\eta}_{2j+1}) \in \bar{\FT}(\mathsf{Com})$.  Finally, projecting $d(\bar{\eta}_{2j+1})$ along $\bar{\FT}(\mathsf{Com})(2j+1,0) \twoheadrightarrow \Sigma^{-4j-2}\mathsf{GC}_2^{2j+1}$, we find 
$d(\bar{\eta}_{2j+1})(\omega_{2j+1}) \neq 0$ from Lemma $\ref{notzero}$.

{\bf Proof of Corollary $\ref{dualcor}$.}  The wheel graph $\omega_{2j+1} \in \Sigma^{4j+2}\mathsf{GC}_2^\ast$ is a cycle, it remains to see that it can't be a boundary.  Suppose that it were a boundary, so that $d_{\mathsf{GC^\ast_2}}(\xi)=\omega_{2j+1}$.  The vector $\xi$ has a canonical pre-image $\xi\in\hat{\mathsf{B}}(\wscmo)$ for which $\partial(\xi) =\partial_1(\xi)+\partial_{>1}(\xi)= \omega_{2j+1}+$ terms of higher internal degree. Since $\partial^2(\xi)=0$, it must be the case that the non-zero scalar multiple of $\beta_{2j+1}$ appearing in $\partial(\omega_{2j+1})$ (after Lemma $\ref{notzero}$) is canceled out by another term in $\partial(\xi)$.  However Corollary $\ref{deg0}$ shows that this can't happen.  We conclude $\omega_{2j+1}$ is not a boundary.

%\bibliography{modularbib}
%\bibliographystyle{alpha}

\end{document}